\documentclass{article}
\usepackage{amsfonts,amsmath,amssymb}
\usepackage{slashed,setspace,graphicx}
\usepackage[all]{xy}
\usepackage{enumerate}
\usepackage[active]{srcltx}
\usepackage{fullpage}
\newtheorem{ithm}{Theorem}

\newtheorem{theorem}{Theorem}[section]
\newtheorem{definition}[theorem]{Definition}
\newtheorem{lemma}[theorem]{Lemma}
\newtheorem{remark}[theorem]{Remark}
\newtheorem{corollary}[theorem]{Corollary}
\newtheorem{proposition}[theorem]{Proposition}

\newtheorem{notation}[theorem]{Notation}
\newenvironment{proof}[1][Proof]{\par\addvspace{2mm}\noindent\textbf{#1.} }{\ \rule{0.5em}{0.5em}\par\vspace{4mm}}

\newenvironment{aknowledgments}{\par\addvspace{4mm}\noindent
{\it Aknowledgments. }\rm}{\par\vspace{1mm}}
\def\leftnote#1{\vadjust{\setbox1=\vtop{\hsize 20mm\parindent=0pt\bf\baselineskip=9pt\rightskip=2mm plus 2mm#1}\hbox{\kern-20mm\smash{\box1}}}}
\newcommand{\Q}{\mathbb{Q}}
\newcommand{\C}{\mathbb{C}}
\newcommand{\Z}{\mathbb{Z}}
\newcommand{\A}{{\mathcal A}}

\newcommand{\N}{{\mathcal N}}
\newcommand{\R}{{\mathcal R}}
\newcommand{\I}{\mathcal{I}}
\newcommand{\U}{\mathcal{U}}
\newcommand{\T}{\mathcal{T}}
\newcommand{\QQ}{{\mathcal Q}}
\newcommand{\LL}{{\mathcal L}}
\newcommand{\PP}{{\mathcal P}}
\newcommand{\D}{{\mathfrak D}}

\newcommand{\pl}{{\mathfrak l}}
\newcommand{\f}{{\mathfrak f}}

\newcommand{\bo}{{\mathfrak O}}
\newcommand{\p}{{\mathfrak p}}

\newcommand{\vareps}{\varepsilon}
\newcommand{\dsp}{\displaystyle}

\newcommand{\wt}{\widetilde}
\newcommand{\wh}{\widehat}
\newcommand{\normres}[3]{\N_{#1}(#2\,|\,#3)}
\def\RT_#1{{\R\T_{#1}^\star}}
\DeclareMathOperator{\Ind}{Ind}
\DeclareMathOperator{\ind}{ind}
\DeclareMathOperator{\coinf}{coinf}

\DeclareMathOperator{\Gal}{Gal}

\DeclareMathOperator{\car}{char}
\DeclareMathOperator{\ver}{ver}
\DeclareMathOperator{\Det}{Det}
\DeclareMathOperator{\Dab}{det}
\DeclareMathOperator{\Hom}{Hom}
\DeclareMathOperator{\ord}{ord}
\DeclareMathOperator{\Cl}{Cl}
\def\isom{\buildrel\sim\over{\longrightarrow}}
\begin{document}
\bibliographystyle{amsalpha}
\title{Self-Dual Integral Normal Bases and Galois Module Structure}

\author{Erik Jarl Pickett and St\'ephane Vinatier}
%
%

\maketitle
\begin{abstract} Let $N/F$ be an odd degree Galois extension of number
  fields with Galois group $G$ and rings of integers $\bo_N$ and
  $\bo_F=\bo$ respectively. Let $\A$ be the unique fractional
  $\bo_N$-ideal with square equal to the inverse different of
  $N/F$. Erez has shown that $\A$ is a locally free $\bo[G]$-module if
  and only if $N/F$ is a so called weakly ramified extension. There
  have been a number of results regarding the freeness of $\A$ as a
  $\Z[G]$-module, however this question remains open. In this paper we
  prove that $\A$ is free as a $\Z[G]$-module assuming that $N/F$ is
  weakly ramified and under the hypothesis that 
for every prime $\wp$ of $\bo$ which ramifies wildly in $N/F$,
the decomposition group is abelian, the ramification
  group is cyclic and $\wp$ is unramified in $F/\Q$.

 We make crucial use of a construction due to the first named author
 which uses Dwork's exponential power series to describe self-dual
 integral normal bases 
in Lubin-Tate extensions of local fields. 
This yields a new and striking relationship between the local
norm-resolvent and Galois Gauss sum involved.
Our results generalise work of
 the second named author concerning the case of base field $\Q$. 
\end{abstract}

\section{Introduction}\label{intro}
Let $N/F$ denote an odd degree Galois extension of number fields. By
Hilbert's formula for the valuation of the different $\D$ of $N/F$,
there exists a fractional ideal $\A$ $(=\A_{N/F})$ of the ring of
integers $\bo$ $(=\bo_{F})$ of $F$ such that:
\[\A^2=\D^{-1}\enspace.\]
This ideal is known as the \textit{square root of the
  inverse different}. It is an ambiguous ideal, namely it is stable
under the action of 
the Galois group $G$ of $N/F$ and hence an $\bo[G]$-module. Erez
has shown $\A$ to be locally free if and only if $N/F$ is weakly ramified,
\textit{i.e.}, if the second ramification group of any prime ideal $\p$ of
$\bo_N$ is trivial. The study of $\A$ as an
$\bo[G]$-module has too many obstructions to be dealt with (in
particular $\bo$ may not be principal), so after
Fr\"ohlich, Taylor, \textit{et al.} (for the Galois module structure
of the ring of integers in a tame extension), we consider the structure of $\A$ as a
$\Z[G]$-module. 
 
In \cite{Erez2} Erez proves that when $N/F$ is tamely ramified, then $\A$ is always free over $\Z[G]$. 
The question of whether $\A$ is free as a $\Z[G]$-module when $N/F$ is
wildly but weakly ramified is still open. In this paper we prove the
following global result. 
\begin{ithm}\label{THM}
Let $N/F$ denote an odd degree weakly ramified Galois extension of
number fields and suppose that, for any wildly ramified prime $\p$ of
$\bo_N$, the decomposition group is abelian, the ramification
group is cyclic and the localised extension $F_\wp/\Q_p$ is unramified
--- where $\wp=\p\cap F$ and $p\Z=\p\cap\Q$. Then $\A$ is a free
$\Z[G]$-module. 
\end{ithm}

This result generalises \cite[th\'eor\`eme 1.2]{Vinatier_jnumb}, which
is the natural analogue in the absolute case $F=\Q$. In that case,
ramification groups at wildly ramified places are always cyclic of
prime order. In the relative case, we will see that the abelian
decomposition group assumption yields that the ramification
group is $p$-elementary abelian at a wildly ramified place above a
rational prime $p$. Therefore, our hypothesis about 
ramification groups at wildly ramified places really mimics the
situation in the absolute case. 

As in Taylor's celebrated theorem for rings of integers in tame
extensions of number fields \cite[Theorem 1]{tay1}, our result
indicates a deep connection between two kinds of invariants of the
extension $N/F$: the Galois Gauss sum, of analytic nature, that
emerges in the constant of the functional equation of the Artin
$L$-function of $F$; and the norm-resolvents attached to semi-local
normal basis generators of $\A$, that are entirely of algebraic nature
\textit{a priori}. 

Generalising the link between these objects to the relative situation
has been made possible by the exhibition by the first named author in
\cite{Pickett} of explicit normal basis generators for cyclic weakly
ramified extensions of an unramified extension of $\Q_p$. These
generators are constructed using values of Dwork's $p$-adic
exponential power series at certain units and have very nice
properties --- in particular they are self-dual with respect to the
trace form. Of course, the norm-resolvents we attach to them can no
longer be thought of as completely algebraic in nature.

Dwork originally introduced his power series in the
context of $p$-adic differential operators, when considering the zeta
function of a hypersurface \cite{Dwork}. In this paper we  
demonstrate that Dwork's power series is extremely useful when 
considering Galois module structure in extensions of both local and 
number fields. We hope that this work will lead to the further 
investigation of the connections between these two subject areas.

The core of this paper is Section
\ref{computation}, where we use the rich properties of Pickett's basis
generators to compute the product of a 
norm-resolvent with a modified twisted Galois Gauss sum in local cyclic wildly
and weakly ramified extensions. We obtain the following local result
(the objects will be defined below).
\begin{ithm}\label{local}
Let $p\not=2$ be a rational prime, $K$ an unramified finite extension of
$\Q_p$, $M$ a cyclic wildly and weakly ramified extension
of $K$ such that $p$ belongs to the norm group of $M/K$.
There exist a normal basis
generator $\alpha_M$ of the square root of the inverse different of
${M/K}$ and choices in the definitions of the norm-resolvent
$\N_{K/\Q_p}(\alpha_M\mid\,.\,)$ and of the modified Galois
Gauss sum $\tau_K^\star$ such that, for any character $\chi$ of $\Gal(M/K)$: 
$$\N_{K/\Q_p}(\alpha_M\mid\chi)\tau_K^\star(\chi-\psi_2(\chi))=1\enspace,$$
where $\psi_2$ is the second Adams operator.
\end{ithm}
Before proving this result in Section \ref{computation}, we introduce
the technical tools for our study in Section \ref{strategy}. Then in
Section \ref{preliminary} we give some preliminary results and explain
how to reduce the proof of Theorem \ref{THM} to that of Theorem
\ref{local}. 
\smallskip

We hope to deal with the general relative abelian case in a future
publication, but we have no explicit description of a normal basis
generator which lends itself so well to calculations of the type used
in this paper --- see \cite[Remark 13(2)]{Pickett} and
\cite[Introduction]{Pickett2}.  
We are therefore not able to generalise the explicit
computations of Section \ref{computation} at this stage. 

Throughout this paper, $N/F$ is an odd degree weakly ramified extension
of number fields with Galois group $G$. 
\begin{aknowledgments}
The authors would like to express their deep gratitude to
Philippe Cassou-Nogu\`es and Martin J. Taylor, for many 
useful comments and suggestions; specifically Philippe
Cassou-Nogu\`es  pointed out an embarrassing misuse of Fr\"ohlich's 
Hom-description in the first version of this paper, and helped solve
the problem. They also thank R\'egis Blache for an enlightening
discussion about $p$-adic analysis.
\end{aknowledgments}
\section{Strategy}\label{strategy}
In this section, we first explain briefly how Fr\"ohlich's
Hom-description translates the problem of showing that $\A$ is a free
$\Z[G]$-module into the study of an equivariant morphism on the group
of virtual characters of $G$, with idelic values. For each rational
prime $p$, the local components above $p$ of this morphism decompose
as a product of factors indexed by the prime ideals of $\bo$. We
recall from the literature the properties we need about these factors,
except for the $\wp$-factors of the $p$-component when $\wp\mid p$ is
wildly ramified in $N/F$; this will be dealt with in Sections
\ref{preliminary} and \ref{computation}.

We first fix some notations for the paper. 
\begin{notation}\label{notation}
We let $\Q^c$ denote the algebraic closure in the field of complex
numbers of the field $\Q$ of rational numbers; for any rational prime
$p$ we fix an algebraic closure $\Q_p^c$ of the field of $p$-adic
numbers $\Q_p$. Any number field (resp. finite extension
of $\Q_p$) $L$ we consider is assumed to be contained in $\Q^c$
(resp. $\Q_p^c$), and we set $\Omega_L=\Gal(\Q^c/L)$
(resp. $\Omega_L=\Gal(\Q_p^c/L)$); we let $L^{ab}$ be the maximal
abelian extension of $L$ in $\Q^{c}$ (resp. $\Q_p^{c}$) and
denote $\Gal(L^{ab}/L)$ as $\Omega_L^{ab}$. 

When $L$ is a number field, we denote by $\bo_L$ its ring of integers
and, if $\p$ is a prime ideal of $\bo_L$, by $L_\p$ the completion
(also called the ``localisation'') of $L$ at $\p$. When $L$ is a finite
extension of $\Q_p$, we denote by $\bo_L$ the valuation ring of $L$ and by $\theta_L$
the Artin reciprocity map $L^\times\rightarrow\Omega_L^{ab}$.

\end{notation}

\subsection{The class group}\label{sub:hom-des}
To prove Theorem \ref{THM} we use the classic strategy developed by
Fr\"ohlich. We associate to the $\Z[G]$-module $\A$, its class $(\A)$
in the class group of locally free $\Z[G]$-modules $\Cl(\Z[G])$.
Since the order of $G$ is odd, the triviality of the class $(\A)$ is
equivalent to $\A$ being free as a $\Z[G]$-module, which is our
goal. Fr\"ohlich's Hom-description of
$\Cl(\Z[G])$ reads:
$$\Cl(\Z[G])\cong\frac{\dsp\Hom_{\Omega_\Q}(R_G,J(E))}{\dsp\Hom_{\Omega_\Q}(R_G,E^{\times\!})\Det({\mathcal U}(\Z[G]))}\enspace.$$
Here $R_G$ is the additive group of virtual characters of $G$ with
values in $\Q^c$, $E$ is a ``big enough'' number field (in particular
$E$ is Galois over $\Q$, contains $N$ and the values of the
elements of $R_G$) and $J(E)$ is its id\`ele group. The
homomorphisms in $\Hom_{\Omega_\Q}(R_G,J(E))$ are those which commute
with the natural actions of $\Omega_\Q$ on $R_G$ and $J(E)$. The group
$\Hom_{\Omega_\Q}(R_G,E^\times)$ embeds in the former one through the
diagonal embedding of $E^\times$ in $J(E)$, and ${\mathcal U}(\Z[G])={\mathbb 
  R}[G]^{\times\!}\times\prod_l\Z_l[G]^{\times\!}$, with $l$ running over
all rational primes. We now briefly define the $\Det$ morphism, as
well as local and semi-local resolvents and norm-resolvents. For a
more complete account of Fr\"ohlich's Hom-description, see
\cite{Frohlich-Alg_numb}. 

\subsection{Determinants and resolvents}\label{defs}
Let $B/A$ be a finite Galois extension of number fields, let $\p$ be a
prime ideal of $\bo_B$ and set $\wp=\p\cap A$. In the following, the
symbols $K$, $L$, $H$, $\Q_*$ and $R$ may have two different meanings
corresponding respectively to the semi-local and local situations:
$$
  \vbox{\halign{%
      \strut\vrule\ $#$ \vrule &&\ $#$ \vrule\cr
      \noalign{\hrule}
         K & L & H & \Q_* & R \cr
      \noalign{\hrule height1pt}
         A & B\otimes_A A_\wp & \Gal(B/A) & \Q & R_H \cr
      \noalign{\hrule}
         A_\wp & B_\p & \Gal(B_\p/A_\wp) & \Q_p & R_{H,p} \cr
      \noalign{\hrule}}}
$$
In the table, we have denoted by $R_{H,p}$ the group of virtual
  characters of $H$ with values in $\Q_p^c$.
Let $\chi$ be the character of an irreducible matrix representation
$\Theta$ of $H$ and $x=\sum_{h\in H}x_hh\in L[H]$, then
$$\Det_\chi(x)=det\Big(\sum_{h\in H}x_h\Theta(h)\Big)\enspace,$$
where $det$ stands for the matrix determinant. Extending this formula
by linearity to any $\chi\in R$ yields the morphism
$$\Det:L[H]^\times\longrightarrow\Hom_{\Omega_K}\big(R,(\Q_*^c)^\times\big)\enspace.$$
The restriction of $\Det_\chi$ to $H$ yields an
abelian character of $H$ that we denote by $\Dab_\chi$.
It can be extended to $\Omega_K$ by letting 
$\Dab_\chi(\omega)=\Dab_\chi(\omega|_L)$ for any $\omega\in\Omega_K$. 

In order to find a representative character function of the image of
$(\A)$ under the Hom-description of $\Cl(\Z[G])$, one needs to
consider resolvents and norm-resolvents. 
\begin{definition}\label{resnormres}
Let $\alpha\in L$ and $\chi\in R$.
The resolvent, and respectively the norm-resolvent, of
$\alpha$ at $\chi$ with respect to $L/K$ are defined as  
$$(\alpha\,|\,\chi)=(\alpha\,|\,\chi)_{{H}}=\Det_\chi\left(\sum_{h\in H}\alpha^hh^{-1}\right)\ ,\quad \mathcal{N}_{K/\Q_*}(\alpha\,|\,\chi)=\prod_{\omega\in\Omega}(\alpha\,|\,\chi^{\omega^{-1}})^\omega$$
where the product is over a (right) transversal $\Omega$ of $\Omega_K$
in $\Omega_{\Q_*}$.  
\end{definition}
Notice that the norm-resolvent depends on the choice of the right
transversal $\Omega$. In view of \cite[Prop.
  I.4.4(ii)]{Frohlich-Alg_numb}, changing $\Omega$ multiplies
the norm-resolvent by $\Dab_\chi(h)$ for some $h\in H$, namely by an
element in the denominator of the Hom-description. It follows that
when using norm-resolvents to describe a representative function for
$(\A)$, we may choose $\Omega$ freely. When $K/\Q_p$ is Galois (in
the local context), restriction of $\Q_p$-automorphisms to $K$ maps any 
such $\Omega$ onto $\Gal(K/\Q_p)$. 
When $H$ is abelian, the formulas simplify to:
$$(\alpha\,|\,\chi)=\sum_{h\in{H}}\alpha^h\chi(h^{-1})\ ,\quad \mathcal{N}_{K/\Q_*}(\alpha\,|\,\chi)=\prod_{\omega\in\Omega}\left(\sum_{h\in{H}}\alpha^{h\omega}\chi(h^{-1})\right)\enspace.$$
\subsection{A representative for $(\A)$}\label{representative_subsection}
We now describe a representative $f$ of $(\A)$ in
$\Hom_{\Omega_\Q}(R_G,J(E))$. Such a representative is not unique
since it can be modified by multiplication by any element in the
denominator of the Hom-description. Indeed we construct $f$ by
slightly modifying Erez's representative $v_{N/F}$ \cite[Theorem 3.6]{Erez2}, in
order to enable more precise computations at wildly ramified places
--- the goal being to show that $f$ itself lies in the denominator of the
Hom-description.  

We define the representative morphism $f$ of $(\A)$ by giving,
for each rational prime $p$, its semi-local component $f_p$ taking
values in $J_p(E)=\prod_{\PP\mid p}E_\PP^\times$, where the product is
over the prime ideals of $\bo_E$ above $p$. This group embeds in
$J(E)$ as the subgroup consisting of the id\`eles $(y_{\PP})_{\PP}\in
J(E)$ such that $y_{\PP}=1$ if $\PP$ is a prime ideal of $\bo_E$ that
does not divide $p$. It decomposes into the cartesian product
$J_p(E)=\prod_{\wp|p}J_\wp(E)$, where the product is
over the prime ideals $\wp$ of $\bo$ above $p$, and
$J_\wp(E)=\prod_{\PP\mid\wp}E_\PP^\times$.  
Note that, with similar definitions at a lower level,
$J_\wp(F)=F_\wp^\times$ diagonally embeds into $J_\wp(E)$ and
$J_\wp(N)$ embeds into $J_\wp(E)$ by the map
$(x_\p)_{\p|\wp}\mapsto(y_\PP)_{\PP|\wp}$, such that $y_\PP=x_\p$ if $\PP|\p$.

Furthermore $J_\wp(E)$ is isomorphic to $(E\otimes_F F_\wp)^\times$
via the isomorphism 
$$\I_\wp=\prod_{\iota}(\iota\otimes 1):(E\otimes_F F_\wp)^\times\isom
J_\wp(E)
\enspace,$$
built on the various embeddings $\iota$ of $E$ in $\Q_p^c$ that fix
$\wp$. These embeddings are in one-to-one correspondence with the
prime ideals of $\bo_E$ above $\wp$. We may not always
distinguish between $J_\wp(E)$ and $(E\otimes_F F_\wp)^\times$ in the
following.

First consider the case where $p$ is a rational prime that does not divide
the order of $G$. Under this assumption, $\Z_p[G]$ is a maximal order
in $\Q_p[G]$, which by \cite[Prop. I.2.2]{Frohlich-Alg_numb}
implies    
$$\Hom_{\Omega_\Q}(R_G,\U_p(E))=\Det(\Z_p[G]^\times)\enspace,$$
where $\U_p(E)=\prod_{\PP\mid p}\bo_{E_\PP}^\times$. On the other
hand, Erez has shown that his representative $v_{N/F}$ takes values in
$\U(E)=\prod_\PP\bo_{E_\PP}^\times$ \cite[Theorem $2^\prime$]{Erez2}. It
follows that $(v_{N/F})_p$ belongs to the $p$-component of the
denominator of the Hom-description, hence we may set $f_p=1$. 
Similar arguments show that we may also set $f_\infty=1$, where
$\infty$ stands for the archimedean place of $\Q$.

From now on we suppose $p$ is a rational prime dividing the order of
$G$. The $p$-component $f_p$ of our representative $f$ is
essentially made of two ingredients: the global Galois Gauss sum of
$F$ and norm-resolvents associated to semi-local generators of $\A$.
\subsubsection{Norm-resolvents.}
We begin with the latter ingredient. Since $N/F$ is weakly ramified,
we know by Erez's criterion 
\cite[Theorem 1]{Erez2} that the square root of the inverse
different $\A$ is locally free. Specifically, for each prime ideal
$\wp$ of $\bo$ above $p$, there exists $\beta_\wp\in N\otimes_{F}{F_\wp}$ 
such that $\A\otimes_{\bo}\bo_{F_\wp}=\bo_{F_\wp}[G]{\beta_\wp}$. The
semi-local resolvent $(\beta_\wp\,|\,\chi)$, for $\chi\in R_G$, takes
values in $(E\otimes_F F_\wp)^\times$, identified with $J_\wp(E)$
through isomorphism $\I_\wp$, then embedded in $J(E)$. The
norm-resolvent is then obtained as 
$$\N_{F/\Q}(\beta_\wp\,|\,\chi)=\prod_{\omega\in\Omega}(\beta_\wp\,|\,\chi^{\omega^{-1}})^\omega\enspace,$$
where the product is over a (right) transversal $\Omega$ of $\Omega_F$
in $\Omega_{\Q}$. The action of $\Omega$ on $J(E)$ permutes the
semi-local subgroups $J_\wp(E)$ corresponding to prime ideals $\wp$ of
$\bo$ above $p$, hence the norm-resolvent
$\N_{F/\Q}(\beta_\wp\,|\,\chi)$ takes values in $J_p(E)$; if $\QQ$
denotes any prime ideal of $\bo_E$ above $p$,
we denote by $\N_{F/\Q}(\beta_\wp\,|\,\chi)_\QQ$ its component in $E_\QQ^\times$.
Accordingly, we denote by $\beta_p$ the id\`ele in $J_p(N)$ whose
components in the subgroups $J_\wp(N)$, where $\wp$ is above $p$, are the
$\beta_\wp$ introduced above. The norm-resolvent
$\N_{F/\Q}(\beta_p\,|\,\chi)$ belongs to $J_p(E)=\prod_{\QQ\mid
  p}E_\QQ^\times$, with $\QQ$-component:
$$\N_{F/\Q}(\beta_p\,|\,\chi)_\QQ=\prod_{\wp\mid
  p}\N_{F/\Q}(\beta_\wp\,|\,\chi)_\QQ\enspace.$$
\subsubsection{Galois Gauss sums.}\label{ggs}
We now turn to the global Galois Gauss sum, which is a product of
local ones. For each prime ideal $\ell$ of $\bo$, we fix a prime ideal
$\LL$ of $E$ above $\ell$ and we set $\pl=\LL\cap N$,
$l\Z=\LL\cap\Q$. Recall the one-to-one correspondence between prime
ideals of $\bo_E$ above $\ell$ and embeddings of $E$ into $\Q_l^c$
that fix $\ell$, and let $\iota_{\LL/\ell}$ denote the embedding associated to
$\LL$. It induces an isomorphism between the Galois group of the local
extension $N_\pl/F_\ell$ and the decomposition group $G(\ell)$ of
$\LL/\ell$, sending $\gamma\in\Gal(N_\pl/F_\ell)$ to 
$\iota_{\LL/\ell}\,\gamma\,\iota_{\LL/\ell}^{-1}\in\Gal(N/F)$ (the
action on elements is written exponentially, thus to the right, see
\cite[III (2.7)]{Frohlich-Alg_numb}). We may thus identify $G(\ell)$
and $\Gal(N_\pl/F_\ell)$ in the following.  

In Subsection \ref{gauss} we define the local Galois Gauss sum
$\tau_{F_\ell}(\chi)$ when $\chi\in R_{G(\ell)}$ is abelian and at
most weakly ramified  --- for a general definition of the Galois Gauss
sum see for instance \cite[I.5]{Frohlich-Alg_numb}. We also recall
how this Galois Gauss sum is modified at (at most) tamely ramified places,  
and present an analogous modification at wildly and weakly ramified places. We  
denote in both cases the resulting character function by
$\tau_{F_\ell}^\star$. Following Erez \cite[\S3]{Erez2}, we finally
twist our modified Galois Gauss sum using the action of the second
Adams operator $\psi_2$ and get: 
$$T_\ell^\star(\chi)=\tau_{F_\ell}^\star\big(\chi-\psi_2(\chi)\big)\enspace.$$
Set $\wt T_\ell^\star=\Ind_{G(\ell)}^GT_\ell^\star$, namely, by definition
of induction on character functions: 
$$\wt
T_\ell^\star(\chi)=T_\ell^\star(\chi_\ell)=\tau_{F_\ell}^\star\big(\chi_\ell-\psi_2(\chi_\ell)\big)\enspace,$$
where for $\chi\in R_G$, we let $\chi_\ell$ denote the restriction of
$\chi$ to $G(\ell)$. 

We now have for each prime ideal $\ell$ of $\bo$ a function
$\wt T_\ell^\star$ on virtual characters of $G$. We shall see in
Subsection \ref{subsub:twist} that it is in fact almost always
trivial: Equation (\ref{unram}) shows that $\wt T_\ell^\star=1$
whenever $\ell$ is unramified in $N/F$.
Let $S_T$ and $S_W$ be the sets of prime ideals of $\bo$ that are
respectively tamely and wildly ramified in $N/F$; their union
$S$ contains all the prime ideals $\ell$ of $\bo$ such
that $\wt T_\ell^\star$ is non trivial, and we define the global
twisted modified Galois Gauss sum associated to $F$ as:
$$T^\star=\prod_{\ell\in S}\wt T_\ell^\star\quad\in\ \Hom(R_G,E^\times)\enspace.$$
It takes values in $E^\times$, which diagonally
embeds into $J(E)$, henceforth into each $J_p(E)$.
\subsubsection{All together.}
We will prove the following result in Subsection \ref{subsub:twist}.
\begin{proposition}\label{representative}
For any rational prime $p$ and for $\chi\in R_G$ set:
$$f_p(\chi)=T^\star(\chi)\N_{F/\Q}(\beta_p\,|\,\chi)
=T^\star(\chi)\prod_{\wp\mid p}\N_{F/\Q}(\beta_\wp\,|\,\chi)$$
if $p$ divides the order of $G$, $f_p(\chi)=1$ otherwise; furthermore
set $f_\infty=1$. Then $f=(f_p)_p$ is a representative of $(\A)$ in
$\Hom_{\Omega_\Q}(R_G,J(E))$. 
\end{proposition}
\subsection{Localising and cutting into pieces}\label{cut}
We fix a rational prime $p$ dividing the order of $G$. The
$\Omega_\Q$-equivariant component $f_p$ of our representative takes
semi-local values. We first transform it into a character function
with local values, that we then cut into factors that will be dealt
with separately.

We use the localisation procedure described in \cite[II.2 \&
  III.2]{Frohlich-Alg_numb}. Let $\QQ$ be a prime ideal of $\bo_E$
above $p$. The associated embedding $\iota=\iota_{\QQ/p}$ embeds $E$
into $E_\QQ\subset\Q_p^c$. It gives rise to a homomorphism
$E\otimes_\Q\Q_p\rightarrow E_\QQ$, again denoted by $\iota$, and to
an isomorphism $\chi\mapsto\chi^\iota$, of $R_G$ onto $R_{G,p}$, 
the ring of virtual characters of $G$ with values in $\Q_p^c$. We know
by \cite[Lemma II.2.1]{Frohlich-Alg_numb} that it yields an
isomorphism: 
\begin{equation}\label{iota}
\iota^*:
\Hom_{\Omega_\Q}(R_G,J_p(E))\isom\Hom_{\Omega_{\Q_p}}(R_{G,p},E_\QQ^\times)
\enspace,
\end{equation}
defined by $\iota^*(v)(\theta)=v(\theta^{\iota^{-1}})^\iota$, and such
that
\begin{equation}\label{iota-det}
\iota^*\big(\Det(\Z_p[G]^\times)\big)=\Det(\Z_p[G]^\times)\enspace,
\end{equation}
where the left hand side Det group takes semi-local values (on
characters with values in $\Q^c$) whereas the right hand side takes
local values (on characters with values in $\Q_p^c$). However the
ambiguity of the notation should not be a problem thanks to this
isomorphism. 

We now compute $\iota^*(f_p)$. Let $\theta\in R_{G,p}$ and set
$\chi=\theta^{\iota^{-1}}$, then
$$\iota^*(f_p)(\theta)=
\prod_{\ell\in S}T_{\ell}^\star(\chi_{\ell})^\iota\
\prod_{\wp\mid p}\N_{F/\Q}(\beta_\wp\,|\,\chi)^\iota
\enspace.$$
For each $\ell\in S$ we let
$\T_\ell^\star\in\Hom(R_{G(\ell),p},E_\QQ^\times)$ be such that, for
$\phi\in R_{G(\ell),p}$:
$$\T_\ell^\star(\phi)=T_\ell^\star\big(\phi^{\iota^{-1}}\big)^\iota\enspace,$$
then
$T_{\ell}^\star(\chi_{\ell})^\iota=\T_{\ell}^\star(\theta_{\ell})$, since
$\iota:R_G\hookrightarrow R_{G,p}$ commutes with restriction of
characters to $G(\ell)$. 

Recall from
Subsection \ref{ggs} that we have fixed a prime ideal $\PP$ of $\bo_E$
above each prime ideal $\wp$ of $\bo$, and set $\p=\PP\cap N$. We have
not specified the 
semi-local generator $\beta_\wp$ yet, but in view of \cite[Theorem
  1]{Erez2} and \cite[Prop. III.2.1]{Frohlich-Alg_numb} --- that
Erez has already checked to apply to our situation --- we may choose a
local generator $\alpha_\wp\in N_\p^\times$ such that
$\A_{N_\p/F_\wp}=\bo_{F_\wp}[G(\wp)]\alpha_\wp$, and set:  
$$(\beta_\wp)_\p=\alpha_\wp\ ,\quad (\beta_\wp)_{\p'}=0\enspace,$$
for any prime ideal $\p'$ of $\bo_N$ above $\wp$ and distinct from
$\p$. It then follows from \cite[Theorem 19]{Frohlich-Alg_numb} (see
also \cite[Prop. 5.1]{Erez2}) that, for $\chi\in R_G$:
$$
\normres{F/\Q}{\beta_\wp}{\chi}_\QQ=
\normres{F_\wp/\Q_p}{\alpha_\wp}{\chi_\wp^{\iota}}\,
\det\nolimits_{\chi^{\iota}}(\gamma_\QQ)\enspace,
$$
for some $\gamma_\QQ\in G$ independent of $\chi$. Note that the
$\QQ$-component $B_\QQ$ of our semi-local norm-resolvent
$B=\normres{F/\Q}{\beta_\wp}{\chi}$ equals $B^\iota$ (strictly speaking we
should write $B^{\I_\wp\iota}$ but we have identified $J_\wp(E)$ with
$E\otimes_FF_\wp$, hence omitted $\I_\wp$). Define 
$\R_\wp\in\Hom(R_{G(\wp),p},E_\QQ^\times)$ by  
$$\R_\wp(\phi)=\normres{F_\wp/\Q_p}{\alpha_\wp}{\phi}$$
for $\phi\in R_{G(\wp),p}$. Reordering the factors in $\iota^*(f_p)$
we get: 
$$\iota^*(f_p)(\theta)=
\prod_{\wp'\nmid p}\Ind_{G(\wp')}^G\big(\T_{\wp'}^\star\big)(\theta)\,
\prod_{\wp\mid p}\left(
\Ind_{G(\wp)}^G\big(\R_\wp\T_{\wp}^\star\big)(\theta)\,
\det\nolimits_{\theta}(\gamma_\QQ)
\right)
\enspace.$$  

Most of the factors of $\iota^*(f_p)$ can be dealt with using former
results, already gathered in \cite[\S 4.2]{Vinatier_jnumb} in the
absolute situation. 
\begin{lemma}\label{3outof4}
Let $E_\QQ^0$ and $E_\QQ^1$ denote the maximal subextensions of $E_\QQ$ over
$\Q_p$ which are respectively unramified and tamely ramified.
\begin{enumerate}[(i)]
\item Suppose $\wp'\nmid p$, then
  $\T_{\wp'}^\star\in\Det(\bo_{E_\QQ^0}[G(\wp')]^\times)$;  
\item suppose $\wp\mid p$ and $\wp\notin S_W$, then
  $\R_\wp\T_\wp^\star\in\Det(\bo_{E_\QQ^1}[G(\wp)]^\times)$.  
\end{enumerate}
\end{lemma} 
\begin{proof}
Note first that since $\wt T_{\wp'}^\star=1$
when $\wp'\notin S$, the same holds for $\T_{\wp'}^\star$. If
$\wp'\nmid p$ and $\wp'\in S_T$, our modified Galois Gauss sum 
$\tau_{F_{\wp'}}^\star$ coincides with $\tau_{F_{\wp'}}^*$, the usual
modified Galois Gauss sum, so we may use \cite[Theorem 3]{tay1}
together with \cite[(2-7)]{cnt-adams} (for the twist of the Galois
Gauss sum by the Adams operator). Suppose now that $\wp'\in S_W$. We
shall prove in Lemma \ref{modifOK} below --- see also Subsection
\ref{subsub:twist} --- that $T_{\wp'}^\star$ can
then be replaced by its non modified analogue $T_{\wp'}:\chi\in
R_{G(\wp')}\mapsto\tau_{F_{\wp'}}(\chi-\psi_2(\chi))$, whose
behaviour is controlled by the immediate extension of the result in 
Lemme 4.7 of \cite{Vinatier_jnumb} to the relative case (noticing that 
the only part of Lemme 2.1 which is required in its proof,
``$G_0=G_1$'', remains true, see Remark \ref{H0=H1} below). This
proves assertion \textsl{(i)}.  

%
If $\wp\mid p$ and $\wp\notin S_W$, once again our modified Galois Gauss
sum is the usual one $\tau_{F_\wp}^*$, so assertion \textsl{(ii)} follows
from Lemme 4.3 of \cite{Vinatier_jnumb} whose proof is readily
checked to apply to the current relative situation.
\end{proof}

In Section \ref{preliminary}, we will show, assuming Theorem
\ref{local}, that assertion \textsl{(ii)} of the preceding lemma
also holds when $\wp\in S_W$. We now explain how to deduce Theorem
\ref{THM} from this result. Using the functorial properties of the 
group determinant regarding induction on character functions
\cite[Theorem 12]{Frohlich-Alg_numb}, it yields
$$\iota^*(f_p)\in\Det(\bo_{E_\QQ^1}[G]^\times)\enspace.$$
In view of (\ref{iota}) we know that $\iota^*(f_p)$ is
$\Omega_{\Q_p}$-equivariant. Therefore,
$$\iota^*(f_p)\in\Det(\bo_{E_\QQ^1}[G]^\times)^{\Omega_{\Q_p}}
=\Det(\bo_{E_\QQ^1}[G]^\times)^{\Gal(E_\QQ^1/\Q_p)}
=\Det(\Z_p[G]^\times)
\enspace,$$
using Taylor's fixed point theorem for group determinants
\cite[Theorem 6]{tay1}. It follows by (\ref{iota-det}) that
$f_p\in\Det(\Z_p[G]^\times)$ for every rational prime $p$ dividing the
order of $G$, and this proves Theorem \ref{THM}.
%
\section{Preliminary results}\label{preliminary}
The core of this paper is the study of the remaining factor
$\R_\wp\T_\wp^\star$ when $\wp$ is a prime ideal of $\bo$ which is
wildly ramified in $N/F$ (hence the rational prime $p$ below $\wp$
divides the order of $G$). In this case, because of the assumption
in Theorem \ref{THM}, the local extension is abelian and weakly
ramified. In order to prepare for the extensive study of this factor
in the next section, we devote Subsection \ref{lawre} to the
description of these extensions, using results from Lubin-Tate theory;
in Subsection \ref{gauss} we define the Galois Gauss sum of characters
of Galois groups of such extensions, explain how we modify it and
state some of its properties; finally in Subsection \ref{reduction} we
show how to deduce from Theorem \ref{local} that
$\R_\wp\T_\wp^\star\in\Det(\bo_{E_\QQ^1}[G(\wp)]^\times)$ when $\wp\in S_W$.
\subsection{Local abelian weakly ramified extensions}\label{lawre}
Let $K$ denote a finite extension of $\Q_p$ for some rational prime
$p$, let $d$ denote the residual degree of $K/\Q_p$ and
set $q=p^d$. We intend to make use of Lubin-Tate theory to describe
the wildly and weakly ramified abelian extensions of $K$, so we fix some
(standard) notations. We refer to \cite[\S 3]{Serre-lubintate} or \cite{iwasawa}
for a brief or more detailed exposition of the
theory respectively.

If $\pi$ is a uniformising parameter of $K$ and $n$ a non
negative integer, we denote by $K_{\pi,n}$ the $n$-th division field
associated to $\pi$ over $K$. This is the same notation as in
\cite{Serre-lubintate}, but note that the numbering is different in \cite{iwasawa}.
We set $K_\pi=\cup_{n\ge 1}K_{\pi,n}$. For a positive integer
$s$, we denote by $K_{un}^s$ the unramified extension of $K$ of 
degree $s$ contained in $\Q_p^c$; we let $K_{un}=\cup_{s\ge 1}K_{un}^s$ be the
maximal unramified extension of $K$ in $\Q_p^c$. Recall our notation that $K^{ab}$ is the
maximal abelian extension of $K$ in $\Q_p^c$, with $\Gal(K^{ab}/K)=\Omega_K^{ab}$. Lubin-Tate theory states
that any abelian extension of $K$ is contained in the compositum of
$K_\pi$ and $K_{un}$:
$$K^{ab}=K_\pi K_{un}\enspace.$$
We remark that the fields $K_{\pi,n}$ and $K_{\pi}$ depend on the uniformising parameter $\pi$; yet we have the following result.
\begin{lemma}\label{s_lemma} 
Given uniformising parameters $\pi$ and $\pi'$ of $K$, then for all $n\in\mathbb{N}$, there exists an $s\in\mathbb{N}$ such that
 $K_{\pi,n}K_{un}^s=K_{\pi',n}K_{un}^s$ .
\end{lemma}
\begin{proof}
In \cite[\S3.7]{Serre-lubintate} Serre proves the result $K^{ab}=K_\pi
K_{un}$ by first showing that $K_\pi K_{un}=K_{\pi'} K_{un}$ for all
uniformising parameters $\pi$ and $\pi'$. Following this proof we can
actually replace $K_{\pi}$ (resp. $K_{\pi'}$) with $K_{\pi,n}$
(resp. $K_{\pi',n}$) at every step to give $$K_{\pi,n}
K_{un}=K_{\pi',n}K_{un}\enspace.$$ This means that the compositum
$K_{\pi,n}K_{\pi',n}$ must be contained in $K_{\pi,n} K_{un}$ and
therefore the extension $K_{\pi,n}K_{\pi',n}/K_{\pi,n}$ is
unramified. We know that $[K_{\pi,n}K_{\pi',n}:K_{\pi,n}]$ is finite,
and therefore $K_{\pi,n}K_{\pi',n}=K_{\pi,n}K_{un}^s$ for some
$s$. The result now follows by symmetry. 
\end{proof}
\begin{proposition}\label{tot-un} 
Let $\pi$ be a given uniformising parameter of $K$ and let $L/K$ be a
weakly ramified abelian extension. Then there exist fields $L^{tot}$
and $L^{un}$ such that $L\subseteq L^{tot}L^{un}$, $L^{tot}\subseteq
K_{\pi,2}$ and $L^{tot}L^{un}/L$ and $L^{un}/K$ are unramified. 
\end{proposition}
Note that $K_\pi/K$ is totally ramified, so $L^{tot}\subseteq
K_{\pi,2}$ implies $L^{tot}/K$ totally ramified.
\begin{proof}
From \cite[Theorem 4.1]{Pickett2} we know that there exist fields $\tilde{L}^{tot}$
and $\tilde{L}^{un}$ such that $L\subseteq
\tilde{L}^{tot}\tilde{L}^{un}$, $\tilde{L}^{tot}/K$ is totally
ramified and $\tilde{L}^{tot}\tilde{L}^{un}/L$ and $\tilde{L}^{un}/K$
are unramified. 
This implies that $\tilde{L}^{tot}/K$ is abelian totally and weakly
ramified, so there exists a uniformising parameter $\pi'$ of $K$ such
that $\tilde{L}^{tot}\subset K_{\pi'}$. Further, let $H$ denote its Galois group
and $c$ the valuation of its conductor: set $U_K^0=\bo_K^\times$ and
for any positive integer $n$, $U_K^n=1+\pi^n\bo_K$, then $c$ is
the minimal integer $n\ge 0$ such that $U_K^n$ is contained in the norm group
$N_{\tilde{L}^{tot}/K}((\tilde{L}^{tot})^\times)$. We know that 
$c=\frac{|H_0|+|H_1|}{|H_0|}$ by \cite[Coro. to Lemma
  7.14]{iwasawa}, hence $c\le 2$. Combining Proposition 7.2(ii) and
Lemma 7.4 of \cite{iwasawa} (recall the numbering of division
fields there is different from ours), we get:
$$\tilde{L}^{tot}\subset K_{\pi',2}\enspace.$$ 

From Lemma \ref{s_lemma} we then have an integer $s$ such that 
 $K_{un}^sK_{\pi',2}=K_{un}^sK_{\pi,2}$. Our result then follows by
 taking $L^{un}=\tilde{L}^{un}K_{un}^s$ and
 $L^{tot}=\tilde{L}^{tot}K_{un}^s\cap K_{\pi,2}$. 
\end{proof}
Let $\Gamma^{(n)}=\Gal(K_{\pi,n}/K)$ then the Artin map 
$\theta_K:K^\times\rightarrow\Omega_K^{ab}$ yields an isomorphism
$(\bo_K/\pi^n\bo_K)^\times\cong\Gamma^{(n)}$, so that
$$\Gamma^{(2)}\cong\Gamma^{(1)}\times\Gamma\enspace,$$
where $\Gamma^{(1)}$ is cyclic of order $q-1$ and
$\Gamma=\{\theta_K(1+\pi u),u\in\bo_K/\pi\bo_K\},$ hence
$\Gamma\cong\bo_K/\pi\bo_K$ is $p$-elementary
abelian of order $q$. We state the following fact for future reference.
\begin{proposition}\label{r}
The subextension $M_{\pi,2}$ of $K_{\pi,2}/K$ fixed by $\Gamma^{(1)}$ has 
$$r=\frac{p^d-1}{p-1}=1+p+\cdots+p^{d-1}$$ 
subextensions $M_i$, $1\le i\le r$, of degree $p$ over $K$, each of
which is the fixed subextension of $M_{\pi,2}/K$ by the kernel of an
irreducible character $\chi_i$ of $\Gal(M_{\pi,2}/K)\cong\Gamma$. 
\end{proposition}
\begin{proof}
See for instance \cite[\S2.2]{Vinatier-3}.
\end{proof}
Notice further that $K_{\pi,2}=K_{\pi,1}M_{\pi,2}$ and $M_{\pi,2}$ is the
maximal $p$-extension of $K$ contained in $K_{\pi,2}$.
\begin{corollary}\label{MLun}
In the notations of Proposition \ref{tot-un},
suppose further that $L/K$ is wildly ramified, then the conclusion of Proposition
\ref{tot-un} holds with $L^{tot}\subseteq M_{\pi,2}$. If moreover the
ramification group of $L/K$ is cyclic, then $L^{tot}=M_i$ for an
integer $i\in\{1,\ldots,r\}$.
\end{corollary}
\begin{proof}
Let $H=\Gal(L/K)$ and for $i\ge -1$, let $H_i$ denote the $i$-th
ramification subgroup (in lower notation). One has $H_1\not=H_2=1$
since $L/K$ is wildly and weakly ramified, so $H_0/H_1=1$ by
\cite[IV.2, Coro.$\,$2 to Prop.$\,$9]{Serre}, namely $H_0$ is a
$p$-group. 

Let $L^{tot}$ and $L^{un}$ be as given by Proposition
\ref{tot-un}. They are linearly disjoint over $K$, so
$\Gal(L^{tot}L^{un}/K)$ equals the direct product 
$\Gal(L^{tot}/K)\times\Gal(L^{un}/K)$ and the ramification group of
$L^{tot}L^{un}/K$ equals that of $L^{tot}/K$. Since $L^{tot}/K$ is
totally ramified and $L^{tot}L^{un}/L$
is unramified, this yields (\textit{e.g.} using Herbrand's theorem)
$$\Gal(L^{tot}/K)=\Gal(L^{tot}/K)_0=H_0\enspace.$$ 
Since $H_0$ is a $p$-group, we get that $L^{tot}\subseteq
{K_{\pi,2}}^{\Gamma^{(1)}}=M_{\pi,2}$. Further $\Gal(L^{tot}/K)$ is a quotient
of the $p$-elementary abelian group $\Gamma$, hence has to be of order $p$ if cyclic.
\end{proof}
\subsection{Local weakly ramified Galois Gauss sums}\label{gauss}
Again here $p$ is a fixed rational prime and $K$ is a finite
extension of $\Q_p$. 
Recall from Notation \ref{notation} that
$\Q^c$ denotes the algebraic closure of $\Q$ in the field of complex 
numbers. For any non negative integer $n$, let $\xi_n$ be
the $p^n$-th primitive root of unity in $\Q^c$ given by
$\xi_n=\exp(\frac{2i\pi}{p^n})$, with the standard notations for
complex numbers, in particular $\xi_0=1$; in the sequel we
shall use the notation:
\begin{equation}\label{zetaxi}
\zeta=\xi_1\ ,\quad \xi=\xi_2
\end{equation}
since we will mainly be concerned with these two $p^n$-th roots of
unity. From now on we use the letter $\pi$ to denote a uniformising
parameter of $K$. 
\subsubsection{Abelian Galois Gauss sum.}\label{ab_GGsum}
Let $L$ be a finite abelian 
extension of $K$ with Galois group $H$. 
We denote by $\wh H$ the group of irreducible characters of $H$
with values in $\Q^c$. Any $\chi\in\wh H$ can be seen as a character of 
$K^\times$ using the composition of the Artin map $\theta_K$ of $K$ with the restriction of
automorphisms to $L$: 
$$\theta_{L/K}:K^\times\rightarrow\Omega_K^{ab}\twoheadrightarrow
\Gal(L/K)=H\enspace.$$
We shall also denote by $\chi$ the character of $K^\times$ obtained in
this way. Set $U_K^0=\bo_K^\times$ and $U_K^m=1+\pi^m\bo_K$ for
positive integers $m$. The conductor $\f(\chi)$ of the character $\chi$
is $\pi^m\bo_K$, where $m$ is the smallest integer such that
$\chi(U_K^m)=1$. 

Let $\D_K=\pi^s\bo_K$ denote the absolute different of $K/\Q_p$ 
and let $\psi_K$ denote the standard additive character of $K$, which is defined by
composing the trace $Tr=Tr_{K/\Q_p}$ with the additive
homomorphism $\psi_p:\Q_p\rightarrow(\Q^c)^\times$ such that
$\psi_p(\Z_p)=1$ and, for any natural integer $n$,
$\psi_p(\frac{1}{p^n})=\xi_n$, the $p^n$-th root of unity defined
above. 
The (local) Galois Gauss sum $\tau_K$ is then defined by \cite[II.2
  p.29]{ma}:
$$\tau_K(\chi)=\sum_{x}\chi\left(\frac{x}{\pi^{s+m}}\right)\psi_K\left(\frac{x}{\pi^{s+m}}\right)\enspace,$$
where $m$ is such that $\f(\chi)=\pi^m\bo_K$ and $x$ runs through a
set of representatives of $U_K^0/U_K^m$. 
In particular $\tau_K(\chi)=\chi(\pi^{-s})$ if $\f(\chi)=\bo_K$.

The Galois Gauss sum is now defined on $\wh H$. Since $\wh H$ is a
basis of the free group of virtual characters $R_H$, we extend $\tau_K$
to a function on $R_H$ by linearity:
$\tau_K(\chi+\chi')=\tau_K(\chi)\tau_K(\chi')$. 
Inductivity in degree $0$ then enables one to extend
$\tau_K$ to virtual characters of non abelian extensions of $K$ ---
see \cite[Theorem 18]{Frohlich-Alg_numb}. Note that the conductor
function $\f$ extends to $R_H$ the same way.

Before modifying the abelian Galois Gauss sum, we introduce the
non-ramified part $n_\chi$ of $\chi\in\wh H$ by $n_\chi=\chi$ if
$\chi$ is unramified, $n_\chi=0$ otherwise. The map $\chi\mapsto
n_\chi$ then extends to an endomorphism of the additive group $R_H$ by
linearity ($n_{\chi+\chi'}=n_\chi+n_{\chi'}$). One easily checks that
if $\chi\in R_H$ is such that $\chi=n_\chi$, then:
\begin{equation}\label{unram1}
\tau_K(\chi)=\Dab_{\chi}(\pi^{-s})=\Dab_{\chi}(\D_K^{-1})\enspace.
\end{equation}
\subsubsection{Modification in the tame and weak cases.}\label{subsub:modif}
In this paper, we will only have to consider the case where $L/K$ is tamely or
weakly ramified, namely $H_1$ or $H_2$ is trivial. Since $\theta_{L/K}$  
sends $U_K^m$ to the $m$-th ramification group in the upper numbering
$H^m$ \cite[4.1 Theorem 1]{Serre-lubintate}, and since $H^1=H_1$ and
$H_2=1\Rightarrow H^2=1$, we see that for $\chi\in\wh H$ we will always have
$\pi^2\bo_K\subseteq \f(\chi)$. We shall say that $\chi\in\wh H$ is unramified if
$\f(\chi)=\bo_K$, tamely ramified if $\f(\chi)=\pi\bo_K$ and weakly ramified if
$\f(\chi)=\pi^2\bo_K$. We shall also say that $\chi\in R_H$ is
unramified if $\f(\chi)=\bo_K$ (or equivalently if $\chi=n_\chi$).
\begin{remark}\label{H0=H1}
If $L/K$ is wildly and weakly ramified and abelian, then
$H_0=H_1$ (see the proof of Corollary \ref{MLun}), thus any
$\chi\in\wh H$ is either unramified or weakly ramified. Indeed in the 
abelian case, the tame and ``wild and weak'' situations do not
occur simultaneously.
\end{remark}

We recall how the Galois Gauss sum is modified in the tame
abelian situation. Fix an element $c_{K,1}\in K$ such that
$c_{K,1}\bo_K=\pi\D_K$. Suppose $L/K$ is (at most) tamely ramified and $\chi$ is
a virtual character of $H$. Recall the definition of the non ramified part
$n_\chi$ of $\chi$ above. The (tame) non ramified characteristic of $\chi$ is
$y_{K,1}(\chi)=(-1)^{\deg(n_\chi)}\Dab_{n_\chi}(\pi)$,
and its modified Galois Gauss sum is
$$\tau_K^\star(\chi)=\tau_K(\chi)y_{K,1}(\chi)^{-1}\Dab_\chi(c_{K,1})\enspace.$$
This function is usually denoted by $\tau_K^*$
\cite[IV.1]{Frohlich-Alg_numb}, we changed the shape of the
star in the exponent to stress the fact that the Galois Gauss sum will be
modified in a different (yet similar) way in the wild and weak
case. Indeed the above remark enables us to treat these two cases separately.

Fix an element $c_{K,2}\in K$ such that
$c_{K,2}\bo_K=\pi^2\D_K$. Note that, if $K'$ is a finite Galois
extension of $K$ with uniformising parameter $\pi'$ and Galois group
$H'$, with $H'_0=H'_1$ and $H'_2=1$ (call that property ``purely
weakly ramified''), then $c_{K,2}\bo_{K'}={\pi'}^2\D_{K'}$.

Suppose $L/K$ is wildly and weakly ramified and abelian.
\begin{definition}\label{def:modif}
The (weak) non ramified characteristic of $\chi\in R_H$ is 
$$y_{K,2}(\chi)=(-1)^{\deg(n_\chi)}\Dab_{n_\chi}(\pi^2)\enspace;$$
its modified Galois Gauss sum is
$$\tau_K^\star(\chi)=\tau_K(\chi)y_{K,2}(\chi)^{-1}\Dab_\chi(c_{K,2})\enspace.$$
\end{definition}
Note that since $n_\chi$ is an unramified character,
$y_{K,2}(\chi)$ does not depend on the choice of the uniformising
parameter $\pi$. On the contrary, $\tau_K^\star$ depends on the choice
of the element $c_{K,2}$, unless $\chi$ is unramified; changing
$c_{K,2}$ multiplies $\tau_K^\star$ by an element of $\Det(H_0)$.

For the purposes of this paper we will only need the abelian modified Galois Gauss sum. Nevertheless, this notion extends to non abelian characters as
in the tame situation. First one shows using the results in Subsection
\ref{lawre} that there exists a maximal ``purely weakly ramified''
(see above)
extension $K^{pw}$ of $K$ --- let $R_{(K)}^{pw}$ denote the free
group generated by the characters of the irreducible representations
of $\Gal(K^{pw}/K)$ over $\Q^c$ with open kernel. One then
easily checks that the proof of \cite[Theorem
  29(ii)]{Frohlich-Alg_numb} shows \textit{mutatis mutandis}
that $y_{K,2}$ is fully inductive, \textit{i.e.}, for $K\subseteq K'\subset
K^{pw}$ and $\chi\in R_{(K')}^{pw}$:
$$y_{K'\!,2}(\chi)=y_{K,2}(\ind\chi)\enspace,$$
where $\ind\chi$ is the induced character of $\chi$ in
$R_{(K)}^{pw}$. Since $\tau_K$ is inductive in degree $0$ in the
purely weakly ramified context as well as in the tame one --- see for
instance \cite[II.4 p.39]{ma}, and since the same holds for
$\Dab(c_{K,2})$ --- see the proof of
\cite[Prop. IV.1.1(iv)]{Frohlich-Alg_numb}, the above definition 
yields modified Galois Gauss sums $\tau_{K'}^\star$ on $R_{(K')}^{pw}$
for any $K'/K$ contained in $K^{pw}$, which are inductive in degree $0$. 

We now check that our modification only involves factors in the
denominator of the Hom-description --- see \cite[Theorem
  29(i)]{Frohlich-Alg_numb} for the tame case.
\begin{lemma}\label{modifOK}
Suppose $L/K$ is abelian, of Galois group $H$, and either tamely or
wildly and weakly ramified. 
The map $\tau_K^\star/\tau_K$ belongs to $\Hom_{\Omega_\Q}(R_H,(\Q^c)^\times)\Det(H)$.
\end{lemma}
\begin{proof}
Let $i=1$ or $2$ depending on whether $L/K$ is tamely or weakly
ramified, and let $\chi\in R_H$. Clearly $y_{K,i}$ is
$\Omega_\Q$-equivariant and takes roots of unity values, since
$\Dab_{n_\chi}$ is an abelian character; set
$h_{K,i}=\theta_{L/K}(c_{K,i})\in H$, then
$\Dab_\chi(c_{K,i})=\Dab_\chi(h_{K,i})$. Therefore,
$\chi\mapsto\Dab_\chi(c_{K,i})\in\Det(H)$.  
\end{proof}

Note that when $\chi\in\wh H$ is weakly ramified, we get
$\tau_K^\star(\chi)=\tau_K(\chi)\chi(c_{K,2})$. Further recall from
\cite[\S1]{tate2} that in this case, the Galois Gauss sum is linked to
the local root number $W(\chi)$ by: 
\begin{equation}\label{rootnumber}
\tau_K(\chi)=p^dW({\chi}^{-1})\enspace,
\end{equation}
where $d$ is the residual degree of $K/\Q_p$. We now show the
following very useful property --- see
\cite[Prop. IV.1.1(vi)]{Frohlich-Alg_numb} for the analogous one in
the tame situation.   
\begin{proposition}\label{modif-unr}
Suppose $L/K$ is wildly and weakly ramified with abelian Galois group
$H$, $\chi,\phi\in\wh H$ with $\phi$ unramified, then:
$$\tau_K^\star(\phi)=-1\ ,\quad\tau_K^\star(\phi\chi)=\tau_K^\star(\chi)\enspace.$$  
\end{proposition}
\begin{proof}
One has $\tau_K(\phi)=\phi(\pi^{-s})$, $y_{K,2}(\phi)=-\phi(\pi^{2})$ and
  $\phi(c_{K,2})=\phi(\pi^{2+s})$, hence the result for
  $\tau_K^\star(\phi)$. The result for $\tau_K^\star(\phi\chi)$
  follows when $\chi$ is unramified. Suppose $\chi$ is ramified, thus
  $\f(\chi)=\f(\phi\chi)=\pi^2\bo_K$ and we get, using Formula
  (\ref{rootnumber}) above together with \cite[\S1 Coro. 2]{tate2}:
$$\tau_K(\phi\chi)=p^dW(\phi^{-1}\chi^{-1})=p^d\phi^{-1}(\pi^{2+s})W(\chi^{-1})=\phi(\pi^{-2-s})\tau_K(\chi)\enspace.$$
Further, $y_{K,2}(\chi\phi)=1=y_{K,2}(\chi)$ and
  $\chi\phi(c_{K,2})=\phi(\pi^{2+s})\chi(c_{K,2})$, hence the result
for $\tau_K^\star(\phi\chi)$. 
\end{proof}
More generally, when $L/K$ is abelian and either tamely or wildly and
weakly ramified, and $\phi\in R_H$ is unramified, one checks from
Definition \ref{def:modif}, using (\ref{unram1}), that
$\tau_K^\star(\phi)=(-1)^{\deg(\phi)}$. 
\subsubsection{Twisting.}\label{subsub:twist}
The last step to build the morphism $f$ defined in Subsection
\ref{representative_subsection} is to twist the modified Galois Gauss
sum by $\psi_2$, the second Adams operation, which is the endomorphism
of $R_G$ defined by $\psi_2(\chi)(g)=\chi(g^2)$ for $\chi\in R_G$ and
$g\in G$. The properties of $\psi_2$, together with Lemma
\ref{modifOK} above, enable us to show that $f$ is, as claimed, a
representative of $(\A)$.  
\begin{proof}[Proof of Proposition \ref{representative}]%
One only has to check that the quotient of $f$ by Erez's
representative $v_{N/F}$ \cite[Theorem 3.6]{Erez2} lies in the
denominator of the Hom-description. Choosing the same semi-local
normal basis generators $\beta_\wp$ to define $f$ and $v_{N/F}$ ---
the change of semi-local generator lies in the denominator of the
Hom-description, see \cite[Coro. to Prop. I.4.2]{Frohlich-Alg_numb}
--- this quotient is the global valued morphism on $R_G$ given by:
$$\frac{f}{v_{N/F}}
=\prod_{\ell\in S}\Ind_{G(\ell)}^G\left(
\frac{\tau_{F_\ell}^\star/\tau_{F_\ell}}{\Psi_2(\tau_{F_\ell}^\star/\tau_{F_\ell})}
\right)\enspace,$$
where $\Psi_2$ is the second Adams operator, defined by
$\Psi_2(v)(\chi)=v(\psi_2(\chi))$ if $v$ a character function. Fix an
$\ell\in S$. We know by Lemma \ref{modifOK} that
$\tau_{F_\ell}^\star/\tau_{F_\ell}\in
\Hom_{\Omega_\Q}(R_{G(\ell)},E^\times)\Det\big(G(\ell)\big)$. By
\cite[(2-7)]{cnt-adams}, $\Psi_2$ preserves
$\Det(\Z_l[G(\ell)]^\times)$ for every rational prime $l$, hence 
$\Psi_2\big(\Det\big(G(\ell)\big)\big)\subset\Det(\Z[G(\ell)]^\times)$;
further $\psi_2$ commutes with the action of $\Omega_\Q$ on
$R_{G(\ell)}$, see \cite[Prop.-Def. 3.5]{Erez2}, hence $\Psi_2$
preserves $\Hom_{\Omega_\Q}(R_{G(\ell)},E^\times)$. Applying
\cite[Theorem 12]{Frohlich-Alg_numb}, we see that $f/v_{N/F}$
belongs to $\Hom_{\Omega_\Q}(R_G,E^\times)\Det(\U(\Z[G]))$ as
required. 
\end{proof}
\subsubsection{Alternative expressions of the twisted modified Galois
  Gauss sum.}
We get back to the previous setting, so $L/K$ is an abelian weakly
ramified extension of $p$-adic fields, of Galois group $H$. The
abelian hypothesis yields that if $\chi\in\wh H$, then
$\psi_2(\chi)=\chi^2$, hence  
$\tau_K^\star(\chi-\psi_2(\chi))=\tau_K^\star(\chi)/\tau_K^\star(\chi^2)$.
If further $\chi$ is weakly ramified, then
\begin{equation}\label{tggs}
\tau_K^\star(\chi-\psi_2(\chi))=\chi(c_{K,2})^{-1}\tau_K(\chi-\chi^2)\enspace.
\end{equation}
Note also that, for any $\chi\in R_H$, $\psi_2(\chi)$ has the same
degree as $\chi$. If $\chi$ is unramified, it follows that
\begin{equation}\label{unram}
\tau_K^\star(\chi-\psi_2(\chi))=1\enspace.
\end{equation}

Recall that $\psi_K$ is the standard additive character of $K$ defined
in Subsection \ref{ab_GGsum} (and beware that $\psi_2$ and $\psi_K$ are
very different functions). 
\begin{proposition}\label{tau}
Suppose $\chi\in\wh H$ is weakly ramified, then there exists $c_\chi\in K$ such that:
\begin{equation}\label{conditions}
c_\chi\,\bo_K=\f(\chi)\D_K\quad\mbox{ and }\quad\forall
y\in\pi\bo_K\,,\ {\chi(1+y)}^{-1}=\psi_K(c_\chi^{-1}y)\enspace;
\end{equation}
further for any such $c_\chi$:
$$\tau_K^\star(\chi-{\chi}^2)=\chi\big(\textstyle\frac{c_\chi}{4c_{K,2}}\big)\psi_K(c_\chi^{-1})^{-1}\enspace.$$
\end{proposition}
\begin{proof}
This result is a direct consequence of \cite[\S1]{tate2}. Using
Formula (\ref{rootnumber}) and applying Proposition 1
of \cite[\S 1]{tate2} to $\alpha={\chi}^{-1}$ and ${\mathfrak
  a}=\pi\bo_K$, we get the existence of $c_\chi\in K$ satisfying conditions
(\ref{conditions}) above 
and that, for any such $c_\chi$: 
$$\tau_K(\chi)=p^d\chi(c_\chi^{-1})\psi_K(c_\chi^{-1})\enspace.$$
It only remains to notice that $c_\chi$ satisfies conditions (\ref{conditions}) for $\chi$
if and only if $c_\chi/2$ satisfies them for ${\chi}^2$, to get the
result through Formula (\ref{tggs}).
\end{proof}
We deduce the result which will be required later.
\begin{corollary}\label{tauQ}
Suppose $K/\Q_p$ is unramified and set
$v_{K,2}=p^2/c_{K,2}\in\bo_K^\times$. Suppose $\chi\in\wh H$ is weakly
ramified, then there exists $v_\chi\in\bo_K^\times$ such that $\forall
u\in\bo_K$, ${\chi(1+up)}^{-1}=\zeta^{Tr(uv_\chi)}$; under this condition,
one has:
$$\tau_K^\star(\chi-{\chi}^2)=\chi\big(\textstyle\frac{v_{K,2}}{4v_{\chi}}\big)\xi^{-Tr(v_\chi)}\enspace.$$
Further:
\begin{enumerate}[(i)]
\item $v_\chi$ is uniquely defined modulo $p$, but $v_\chi+ap$ also
  satisfies the above condition for any $a\in\bo_K$; 
\item if $j\in\{1,\ldots,p-1\}$, $v_{\chi^j}$ can be chosen equal to
  $jv_\chi$.
\end{enumerate}
\end{corollary}
\begin{proof}
Since $K/\Q_p$ is unramified, $p$ is a uniformising parameter of $K$
 and $\f(\chi)\D_K=\f(\chi)=p^2\Z_p$. Let $c_\chi\in K$ satisfying conditions
(\ref{conditions}) for $\chi$ and set $v_\chi=p^2/c_\chi$, then
$v_\chi\in\bo_K^\times$ and for all $u\in\bo_K$ one has
${\chi(1+up)}^{-1}=\psi_K(uv_\chi/p)=\zeta^{Tr(uv_\chi)}$. Conversely,
if $v_\chi$ satisfies these conditions, then $c_\chi=p^2/v_\chi$ satisfies conditions 
(\ref{conditions}) and the formula given in Proposition \ref{tau}
yields the formula for the Galois Gauss sum. 

Let $k$ denote the residue field of $K$.
The trace form $T(x,y)=Tr_{k/(\Z_p/p\Z_p)}(xy)$ is a non degenerate
symmetric bilinear form from $k\times k$ to $\Z_p/p\Z_p$, hence induces
an isomorphism $y\mapsto T(\,.\,,y)$ between $k$ and its dual. If
$v_\chi\in\bo_K^\times$ satisfies the condition of the Corollary, then
$T(\,.\,,v_\chi\mod p)$ is given by this condition, hence $v_\chi\mod p$ is unique.
The two last assertions are readily checked.
\end{proof}
\begin{remark}
Under the hypothesis of Corollary \ref{tauQ} we get:
$$\tau_K^\star\big(\chi-{\chi}^2\big)^p=\zeta^{-Tr(v_\chi)}=\chi(1+p)$$
so the modified twisted Galois Gauss sum is a $p$-th root of unity if and only if
$\chi(1+p)$ is trivial, namely when $\theta_{L/K}(1+p)$ belongs to
$\ker(\chi)$. If $K/\Q_p$ is ramified, we get by Proposition \ref{tau} that
$\chi(1+p)^{-1}=\psi_K(c^{-1}p)=\psi_K(c^{-1})^p$, therefore we have
$$\tau_K^\star(\chi-{\chi}^2)^p=\chi(1+p)=1$$
since $p\in\pi^2\bo_K$. The modified abelian weakly ramified twisted Galois
Gauss sum is thus always a $p$-th root of unity in the ramified base field case.
\end{remark}
\subsection{Reduction of the problem}\label{reduction}
Recall that $p$ is a rational prime dividing the order of $G$ (in
particular $p\not=2$ since $[N:F]$ is odd); $\QQ$ is a prime ideal of 
$\bo_E$ above $p$; $\wp$ is a prime ideal of $\bo$ above $p$ which is
wildly ramified in $N/F$. In Subsection \ref{ggs} we have fixed a
prime ideal $\PP$ of $E$ above $\wp$ and denoted by $\iota_{\PP/\wp}$
the corresponding embedding of $E$ into $\Q_p^c$ fixing $\wp$; by
$G(\wp)$ the image in $G$ of $\Gal(N_\p/F_\wp)$, where $\p=\PP\cap N$,
by the homomorphism induced by $\iota_{\PP/\wp}$. 
We introduce the following useful notations. 
\begin{notation}
If $K$ is a finite extension of $\Q_p$ and
$L/K$ is weakly ramified with Galois group $H$,
we let $\alpha_L$ denote a normal basis generator of the square
root of the inverse different of $L/K$ and $\RT_K(L)$ the morphism
from $R_{H,p}$ to $E_\QQ^\times$ given by
$$\RT_K(L)(\chi)=\N_{K/\Q_p}(\alpha_L\,|\,\chi)\T_{K}^\star(\chi)\enspace.$$
We will sometimes write $\tau_K^\star(\chi-\chi^2)$ instead of
$\T_K^\star(\chi)$, when the context makes it clear what we mean.
\end{notation}
We are thus interested in computing $\RT_{F_\wp}(N_\p)=\R_\wp\T_\wp^\star$.

By the hypothesis in Theorem \ref{THM}, $F_\wp$ is unramified over $\Q_p$
and $N_\p/F_\wp$ is abelian, wildly and weakly ramified, with cyclic
ramification group. We are therefore in a position to apply Corollary
\ref{MLun} with $K=F_\wp$, uniformising parameter $\pi=p$, and
$L=N_\p$. Recall that $M_{p,2}$ has $r=\frac{p^d-1}{p-1}$ subextensions
$M_i$, $1\le i\le r$, of degree $p$ over $K$. We know that there exist
an unramified extension $L^{un}/F_\wp$ and an integer $1\le i\le r$, such that:
$$N_\p\subseteq M_iL^{un}\ \mbox{ and }\ M_iL^{un}/N_\p\mbox{ is unramified.}$$
\begin{proposition}\label{reduction_prop}
Let $C_i=\Gal(M_i/F_\wp)$ then: 
$$\RT_{F_\wp}(M_i)\in\Det(\bo_{E_\QQ^1}[C_i]^\times)\Longrightarrow
\RT_{F_\wp}(N_\p)\,\in\,\Det(\bo_{E_\QQ^1}[G(\wp)]^\times)\enspace.$$ 
\end{proposition} 
Before giving the proof, we describe these extensions in a diagram in
which we introduce
some notations for the Galois groups involved.
$$
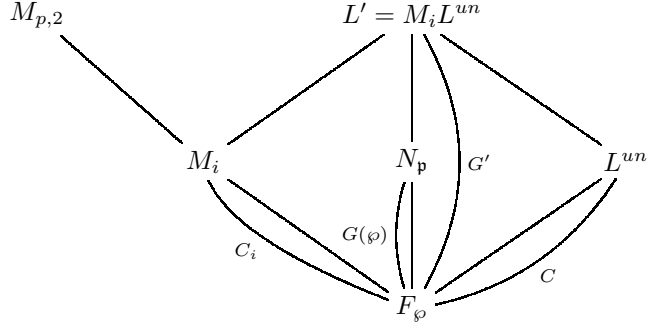
\xymatrix@=40pt@C40pt{
M_{p,2}\ar@{-}@/0pc/[dr] && L'=M_iL^{un}\ar@{-}@/0pc/[dr]\ar@{-}@/0pc/[dl] &\\
& M_i\ar@{-}@/0pc/[dr]\ar@{-}@/l1.5pc/[dr]_{C_i} & N_\p \ar@{-}@/0pc/[u] & L^{un}\ar@{-}@/0pc/[dl] \\
&& F_\wp \ar@{-}@/0pc/[u]\ar@/^.5pc/@{-}[u]^{G(\wp)}\ar@/_1.5pc/@{-}[uu]_{G'}\ar@/_1pc/@{-}[ur]_{C}& \\
}$$
\begin{proof}
There will be two steps in the reduction process: from $M_i$ to
$L'=M_iL^{un}$ and from $L'$ to $N_\p$. For the first step, we shall
take advantage of the fact that $L'/F_\wp$ is the compositum of the
totally ramified extension $M_i/F_\wp$ and the unramified extension
$L^{un}/F_\wp$.
\begin{lemma}\label{step2}
Let $K$ be a finite Galois extension of $\Q_p$ and $K^t$ be its maximal tame
extension in $\Q_p^c$. Let $L_u/K$ be an unramified
extension, $L_1/K$ be an abelian totally wildly and weakly ramified
extension and set $L_2=L_1L_u$, $G_1=\Gal(L_1/K)$ and $G_2=\Gal(L_2/K)$. 
Then: 
$$\RT_K(L_1)\in\Det(\bo_{K^t}[G_1]^\times)\Longrightarrow
\RT_K(L_2)\in\Det(\bo_{K^t}[G_2]^\times)\enspace.$$
\end{lemma}
Notice that the existence of a normal basis generator $\alpha_2$ for
$\A_2$ is ensured by \cite[\S 2 Theorem 1]{Erez2} and \cite[Proposition
  2.2(ii)]{Vinatier_jnumb}. 
\begin{proof}
Since $L_u/K$ is unramified, let us denote by $C$ its cyclic Galois group, 
  then $G_2=G_1\times C$ and any irreducible character $\chi_2$ of
  $G_2$ decomposes as a product $\chi_2=\chi_1\chi_C$, where $\chi_1$ (resp. $\chi_C$) is an
  irreducible character of $G_1$ (resp. $C$). Let $\beta$ denote a
  normal basis generator for $\bo_{L_u}=\A_{L_u/K}$ (over
  $\bo_K$). Since $L_1/K$ and $L_u/K$ are linearly disjoint, one has
  $\A_2=\A_1\otimes_{\bo_K}\bo_{L_u}$. This implies that
  $\alpha_1\otimes\beta$ is a normal basis generator for $\A_2$ over
  $\bo_K$, so there exists $u\in\bo_K[G_2]^\times$ such that
  $\alpha_2=(\alpha_1\otimes\beta)u$. By \cite[\S I, Coro. to
  Prop. 4.2]{Frohlich-Alg_numb}, this yields
\begin{equation}\label{resbasis}
\begin{split}
(\alpha_2\,|\,\chi_2) & = (\alpha_1\otimes\beta\,|\,\chi_2)\Det_{\chi_2}(u)\\
& = (\alpha_1\,|\,\chi_1)(\beta\,|\,\chi_C)\Det_{\chi_2}(u)\\
& = (\alpha_1\,|\,\chi_1)\Det_{\chi_C}(B)\Det_{\chi_2}(u)\enspace,
\end{split}
\end{equation}
where $B=\sum_{c\in C}c(\beta)c^{-1}\in\bo_{L_u}[C]^\times$ by
\cite[Proposition I.4.3]{Frohlich-Alg_numb}. Note that $\chi_C$ is the restriction
to $C$ of $\chi_2$ and that
$[\chi_C\mapsto\Det_{\chi_C}(B)]\,\in\,\Det(\bo_{L_u}[C]^\times)$, so
$[\chi_2\mapsto\Det_{\chi_C}(B)]\,\in\,\Det(\bo_{L_u}[G_2]^\times)$ by
the functorial properties of $\Det$. Consequently, let
$v\in\bo_{L_u}[G_2]^\times$ be such that
$(\alpha_2\,|\,\chi_2)=(\alpha_1\,|\,\chi_1)\Det_{\chi_2}(v)$. 
To compute the norm-resolvent $\N_{K/\Q_p}$, we need to 
choose a transversal $\Omega$ of $\Omega_{K}$ in $\Omega_{\Q_p}$. Since
$\Omega_{\Q_p}/\Omega_K=\Gal(L_u/\Q_p)\,/\,\Gal(L_u/K)$ (note that
$L_u/\Q_p$ is unramified hence Galois), we can choose 
$\Omega$ so that $\Omega|_{L_u}\subset\Gal(L_u/\Q_p)$.
With this choice, $v'=\prod_\Omega
v^\omega\in\bo_{L_u}[G_2]^\times$ is such that:
$$\N_{K/\Q_p}(\alpha_2\,|\,\chi_2)=\N_{K/\Q_p}(\alpha_1\,|\,\chi_1)\Det_{\chi_2}(v')\enspace.$$

We now consider the twisted modified Galois Gauss sum. Since $\chi_C$
is an unramified character, we know by Proposition \ref{modif-unr} that
$\tau_K^\star(\chi_2-\chi_2^2)=\tau_K^\star(\chi_1-\chi_1^2)$.
This yields:
$$\N_{K/\Q_p}(\alpha_2\,|\,\chi_2)\tau_K^\star(\chi_2-\chi_2^2)=\N_{K/\Q_p}(\alpha_1\,|\,\chi_1)\tau_K^\star(\chi_1-\chi_1^2)\Det_{\chi_2}(v')\enspace.$$
Suppose $\RT_K(L_1)\in\Det(\bo_{K^t}[G_1]^\times)$, then by induction of
character functions, the map
$\chi_2\mapsto\N_{K/\Q_p}(\alpha_1\,|\,\chi_1)\tau_K(\chi_1-\chi_1^2)$ 
belongs to $\Det(\bo_{K^t}[G_2]^\times)$, and the same holds for
$\chi_2\mapsto\Det_{\chi_2}(v')$, thus for $\RT_K(L_2)$. 
\end{proof}
Suppose $\RT_{F_\wp}(M_i)\in\Det(\bo_{E_\QQ^1}[C_i]^\times)$. 
We know that $L'=M_iL^{un}$ is unramified over $N_\p$ and that
$N_\p/F_\wp$ is weakly ramified, therefore $L'/F_\wp$ is weakly ramified.
We then apply Lemma \ref{step2} to $K=F_\wp$, $L_1=M_i$,
$L_u=L^{un}$, and get that $\RT_{F_\wp}(L')\in\Det(\bo_{E_\QQ^1}[G']^\times)$.

We consider the restriction of $F_\wp$-automorphisms of $L'$
to $N_\p$: $G'\twoheadrightarrow G(\wp)$. It induces an inflation map on
characters $\inf:R_{G(\wp),p}\rightarrow R_{G',p}$, which in
turn induces a co-inflation map on character functions
$${\coinf}=\coinf_{G(\wp)}^{G'}:\Hom_{\Omega_{\Q_p}}(R_{G',p},E_\QQ^\times)\rightarrow\Hom_{\Omega_{\Q_p}}(R_{G(\wp),p},E_\QQ^\times)\enspace,$$
and we know by \cite[Theorem 12 (ii)]{Frohlich-Alg_numb} that
\begin{equation}\label{coinf_formula}
\coinf\RT_{F_\wp}(L')\in\Det(\bo_{E_\QQ^1}[G(\wp)]^\times)\enspace.
\end{equation} 
We now show the following result.
\begin{lemma}\label{coinf}
Let $K$ be a finite Galois extension of $\Q_p$, $L_2/L_1/K$ be a tower of
abelian extensions such that $L_2/K$ is weakly ramified. Set
$G_1=\Gal(L_1/K)$ and $G_2=\Gal(L_2/K)$. Then
there exists $v\in\bo_K[G_1]^\times$ such that 
$$\coinf\RT_K(L_2)=\RT_K(L_1)\Det(v)\enspace.$$
\end{lemma}
\begin{proof}
For $\chi\in R_{G_1,p}$, one has
\begin{align*}
{\coinf}\,\RT_K(L_{2})\,(\chi) &= \RT_K(L_2)(\inf\,\chi)\\
&= \N_{K/\Q_p}(\alpha_2\,|\,\inf\,\chi)\tau_K^\star\big(\inf\chi-(\inf\chi)^2\big)\enspace.
\end{align*}
We know that the Galois Gauss sum is inflation invariant --- use
(\ref{rootnumber}) and see \cite[p.18 and 22]{ma}, \textit{i.e.}, 
$\tau_K(\inf\chi)=\tau_K(\chi)$. The same clearly holds for the twisted Galois
Gauss sum, as well as for its modified version, thanks to (\ref{tggs})
and since:
$$(\inf\chi)(c_{K,2})=\chi\big(\theta_{L_2/K}(c_{K,2})|_{L_1}\big)=\chi\big(\theta_{L_1/K}(c_{K,2})\big)=\chi(c_{K,2})\enspace.$$
For the resolvent, we know by Lemma 1.5 of
\cite[III]{Frohlich-Alg_numb}, which is readily checked to apply to non tame extensions:
$$(\alpha_2\,|\,\inf\,\chi)_{G_2}=\big(Tr_{L_2/L_1}(\alpha_2)\,|\,\chi\big)_{G_1}\enspace,$$
where the subscripts stress the fact that the sums defining both
resolvents are not indexed by the same group.
Further since ${\A}_1=Tr_{L_2/L_1}({\A}_2)$ (see \cite[\S 5]{Erez2}),
$Tr_{L_2/L_1}(\alpha_2)$ is a normal basis generator for ${\mathcal
  A}_1$, so there exists some $u\in\bo_K[G_1]^\times$ such that
$Tr_{L_2/L_1}(\alpha_2)=u\alpha_1$. As in formula (\ref{resbasis}), this yields:
$$\big(Tr_{L_2/L_1}(\alpha_2)\,|\,\chi\big)=(\alpha_1\,|\,\chi)\Det_\chi(u)$$
and so, for any transversal $\Omega$ of $\Omega_K$ in
$\Omega_{\Q_p}$, we get:
$$\N_{K/\Q_p}(\alpha_2\,|\,\inf\,\chi)=\N_{K/\Q_p}(\alpha_1\,|\,\chi)\,\Det_\chi(v)\enspace,$$
where
$v=\prod_{\sigma\in\Gal(K/\Q_p)}u^\sigma\in\bo_K[G_1]^\times$.
\end{proof}
Proposition \ref{reduction_prop} now follows from Lemma \ref{coinf}
and Equation (\ref{coinf_formula}). 
\end{proof}

Since we have applied Corollary \ref{MLun} to $F_\wp$ with
uniformising parameter $p$, we know that $M_i\subseteq (F_\wp)_{p,2}$,
thus $p$ belongs to the norm group $N_{M_i/K}(M_i^\times)$ of $M_i/K$
by \cite[Lemma 7.4]{iwasawa}. Assuming Theorem \ref{local} we get that
$\RT_{F_\wp}(M_i)=1$ for appropriate choices of $\alpha_{M_i}$; of the
transverse $\Omega$ of $\Omega_{F_\wp}$ in $\Omega_{\Q_p}$ that
defines the norm-resolvent; and of the element $c_{F_\wp,2}$ of
$\bo_{F_\wp}$ that defines the modified twisted Galois Gauss
sum. Therefore, it follows from Proposition \ref{reduction_prop} that
for $\wp\in S_W$:
$$\RT_{F_\wp}(N_\p)\in\Det(\bo_{E_\QQ^1}[G(\wp)]^\times)$$
as announced. We are left with proving Theorem \ref{local}, which is
the goal of the next section.

\section{The local computation}\label{computation}
We fix a rational prime $p$ and an unramified finite extension $K$
of $\Q_p$; we denote by $k$ its residue field, and set $d=[K:\Q_p]=[k:\Z_p/p\Z_p]$.
Let $M$ be a cyclic wildly and weakly ramified extension of $K$, with
Galois group $H$, such that $p$ belongs to the norm group
$N_{M/K}(M^\times)$ of $M/K$, namely $M\subset
K_p$, thus $M$ is one of the degree $p$ subextensions $M_i$ of $K_{p,2}/K$ of
Proposition \ref{r} applied to $\pi=p$. It follows that $M$
is the fixed subfield of $M_{p,2}/K$ by the kernel of an irreducible
character $\chi$ of $\Gal(M_{p,2}/K)$, and the irreducible
characters of $H$ are the $\chi^j$, $0\le j\le p-1$ ($\chi^0$ is the
trivial character $\chi_0$).  

Let $0\le j\le p-1$, then 
$\RT_K(M)(\chi^j)=\mathcal{N}_{K/\mathbb{Q}_p}(\alpha_{M}\mid\chi^j)\tau_K^\star(\chi^j-\chi^{2j})$. 
We first use the explicit construction from \cite{Pickett} of a
self-dual normal basis generator $\alpha_{M}$ for $\A_{M/K}$ over
$\bo_K$, to calculate the norm-resolvent
$\mathcal{N}_{K/\mathbb{Q}_p}(\alpha_{M}\mid\chi^j)$. We then
calculate the modified twisted Galois Gauss sum
$\tau_K^\star(\chi^j-\chi^{2j})$, and show that $\RT_K(M)=1$ for
appropriate choices of the transverse $\Omega$ defining the
norm-resolvent and the element $c_{K,2}$ of $K$ defining the modified 
Galois Gauss sum.
%
\subsection{Dwork's exponential power series}
Let $\gamma\in\Q_p^c$ be a root of
the polynomial $X^{p-1}+p$ and note that, as this is an Eisenstein
polynomial, $\gamma$ will be a uniformising parameter of $K(\gamma)$. 

\begin{definition}\label{dwork}
We define Dwork's exponential power series as
\[E_{\gamma}(X)=\exp(\gamma X -\gamma X^p),\] where the right hand
side is to be thought of as the power series expansion of the
exponential function.
\end{definition}

Here we recall some important properties of Dwork's power series. Let
$\C_p$ denote the completion of $\Q_p^c$ and
$|\,.\,|_p:\C_p\rightarrow{\mathbb R}$
its absolute value such that $|p|_p=p^{-1}$. For instance:
$$|\gamma|_p=|N_{K(\gamma)/\Q_p}(\gamma)|_p^{1/[K(\gamma):\Q_p]}=|p|_p^{1/[K(\gamma):K]}=p^{-1/(p-1)}\enspace,$$
where $N_{K(\gamma)/\Q_p}$ stands for the norm from $K(\gamma)$ to $\Q_p$.
We denote by $\ord_p$ the associated
valuation: for $a\in\C_p$, $|a|_p=p^{-\ord_p(a)}$.
For $r\in{\mathbb R}$, let $D(r^-)=\{a\in\C_p:|a|_p<r\}$
(resp. $D(r^+)=\{a\in\C_p:|a|_p\le r\}$) denote the so-called
open (resp. closed) disc of radius $r$ about $0$. 

The radius of convergence of a series $\sum_na_nX^n$ with coefficients
in $\C_p$ is $\big(\limsup|a_n|_p^{1/n}\big)^{-1}$; it equals
the largest real number $r$ such that the series converges
in $D(r^-)$, see \cite[IV1]{Koblitz}. From standard theory, we know that the radius of
convergence of $\exp$ is $p^{-1/(p-1)}$. If we write $E_\gamma(X)=\sum_{n\ge 0}e_nX^n$, then
$|e_n|_p^{1/n}\le p^{(1-p)/p^2}$ for all positive $n$ by \cite[14 Lemma
  2.2(i)]{LangII}, therefore the radius of convergence of $E_\gamma$ is at
least $p^{(p-1)/p^2}$. In particular, $E_\gamma$ converges on $D(1^+)$,
hence on $\bo_K$. Note that one has
$E_\gamma(a)=\exp(\gamma a-\gamma a^p)$ 
for $a\in D(1^-)$, because then $|\gamma a-\gamma a^p|_p=|\gamma
a|_p<p^{-1/(p-1)}$, but this expression can not be used  
when $|a|_p=1$: the image of the unit circle of $\C_p$ by
$\gamma X-\gamma X^p$ is not contained in the disk of
convergence of $\exp$. For such an $a$, only the series in
$e_n$'s is available.

Nevertheless, if $a\in D(1^+)$ then $|p\gamma a|_p\le
|p\gamma|_p<p^{-1/(p-1)}$, and the same holds for $|p\gamma
a^p|_p$. Using the homomorphic property of the exponential power
series, we deduce that: 
$$E_\gamma(a)^p=\exp(p\gamma a)\exp(-p\gamma a^p)\enspace.$$
In particular $E_{\gamma}(1)$ is a $p$-th root of unity.
By \cite[14 Lemma 2.2(ii)]{LangII} and the power series expansion of
$\exp$ we get:
\begin{equation}\label{exp}
E_{\gamma}(X)\equiv 1+\gamma X \mod\gamma^2\bo_{\Q_p(\gamma)}[[X]]\enspace.
\end{equation}
This shows that $E_{\gamma}(1)$ is in fact a primitive $p$-th root of
unity. Further the different choices of
$\gamma$ correspond to the different choices of this root of unity;
therefore we may choose $\gamma$ so that $E_{\gamma}(1)=\zeta$, where $\zeta$ was
defined in subsection \ref{gauss}. 
We remark that
$[K(\zeta):K]=[K(\gamma):K]$ and $\zeta\in K(\gamma)$, therefore
$K(\gamma)=K(\zeta)$; we shall now denote this field by $K'$. 

Formula (\ref{exp}) also shows that if we let $u\in\bo_K$ be a unit, then
$E_{\gamma}(u)-1$ is a uniformising parameter in $K'$.  
\medskip

Let $\mu\in\Z_p$ and set
$B_\mu(X)=\sum_{n\ge 0}\frac{\mu(\mu-1)\ldots(\mu-n+1)}{n!}X^n$. This
series belongs to $\Z_p[[X]]$ and converges on $D(1^-)$, see
\cite[p.81]{Koblitz}. For any sequence of rational integers
$(\mu_i)_{i}$ converging towards $\mu$, one has
$B_\mu(X)=\lim_iB_{\mu_i}(X)$ (coefficient-wise). Further the $\mu_i$'s
can be taken to be positive, in which case
$B_{\mu_i}(X)=(1+X)^{\mu_i}$, so we may abbreviate notations
writing $B_\mu(X)=(1+X)^\mu$. Using the fact that
$\exp(X)^{\mu_i}=\exp(\mu_iX)$ for every $i$, and taking the limit of
the coefficients when $i$ goes to infinity, one deduces that:
$$\exp(\mu X)=\exp(X)^\mu\enspace.$$

We consider the power series
$E_\gamma(X)^\mu=B_\mu\big(E_\gamma(X)-1\big)$, and see using
(\ref{exp}) that it converges on $D(1^+)$. Substituting $\mu X$ for $X$
in $E_\gamma$ yields a power series $E_\gamma(\mu X)$ that also
converges on $D(1^+)$. Further, let $\mu_{p-1}$ denote the subgroup of
$\Z_p^\times$ of $(p-1)$-th roots of unity. We get:
\begin{lemma}\label{mup-1}
Let $\mu\in\mu_{p-1}$, then $E_\gamma(\mu X)=E_\gamma(X)^\mu$.
\end{lemma}
\begin{proof}
The result is straightforward since
$$\exp\big(\gamma\mu X-\gamma(\mu X)^p\big)=\exp\big(\mu(\gamma X-\gamma X^p)\big)=\exp(\gamma X-\gamma X^p)^\mu\enspace.$$
\end{proof}

%
\subsection{The Kummer extensions in $K_{p,2}$}
Throughout this section we will sometimes identify the multiplicative
groups of the residue fields $k$ and $\Z_p/p\Z_p$ with their
Teichm\"uller lifts. Namely let $\mu_{q-1}$ and $\mu_{p-1}$ denote the
groups of roots of unity of order prime to $p$ in $K$ and $\Q_p$
respectively, then $k^\times\cong\mu_{q-1}$ and
$(\Z_p/p\Z_p)^{\times}\cong\mu_{p-1}$. Specifically, 
$$\bo_K^{\times}=\mu_{q-1}\times(1+p\bo_K)\quad\text{ and }\quad\Z_p^{\times}=\mu_{p-1}\times(1+p\Z_p)\enspace.$$
Since $K/\Q_p$ is unramified, we shall also identify its Galois group
and that of the residue extension, and set
$\Sigma=\Gal(k/(\mathbb{Z}_p/p\mathbb{Z}_p))=\Gal(K/\Q_p)$. 

We now let $\eta$ be a normal basis generator for $k$
over $\mathbb{Z}_p/p\mathbb{Z}_p$ and as described we will often
think of $\eta$ as actually lying in $\bo_K^{\times}$. The conjugates
of $\eta$ under $\Sigma$ are the $\eta^{p^j}$, $0\le j\le d-1$, so
each $u\in k$ has a unique decomposition:
\[u=\sum_{j=0}^{d-1}u_{j}\eta^{p^j}\enspace,\]
with coefficients $u_{j}\in\Z_p/p\Z_p$. 
For ease of notation we identify $0\in\Z_p/p\Z_p$ and $0\in\bo_K$,
so that each $u_{j}\in\Z_p/p\Z_p$ can be seen as an element
$u_{j}\in\{0\}\cup\mu_{p-1}\subset\Z_p$. To $u\in k$ as above we
associate: 
\begin{equation}\label{x(u)}
x_{u}=\prod_{j=0}^{d-1}E_{\gamma}\big(\eta^{p^j}\big)^{u_j}
=\prod_{j=0}^{d-1}E_{\gamma}\big(u_j\eta^{p^j}\big)\ \in\,K'
\end{equation}
(recall $K'$ contains $\gamma$ and is a complete field). The second
equality comes from Lemma \ref{mup-1}; note that $x_u$ does not equal
$E_\gamma(u)$ since $E_\gamma$ is not a group homomorphism on the
additive group $D(1)^+$. Further
$x_{0}=1$ and when $u\in k^\times$, it follows from Formula
(\ref{exp}) that $x_{u}\equiv 1+\gamma u\mod\gamma^2\bo_{K'}$, so
$x_{u}-1$ is a uniformising parameter in $K'$.   
\begin{proposition}\label{xi} 
There are exactly $r=\frac{p^d-1}{p-1}$ degree $p$ extensions of $K'$
contained in $K_{p,2}$, given by $L_i=K'M_i$, $1\le i\le r$. Further
$k^\times/(\Z_p/p\Z_p)^\times$ is in one-to-one correspondence with the
set $\{L_i,\, 1\le i\le r\}$ \textit{via} the map $\bar u\mapsto
K'\big(x_{u}^{1/p}\big)$. 
\end{proposition}
\begin{proof}
From \cite[Theorem 5]{Pickett}, we know that every degree $p$ 
extension of $K'$ contained in $K_{p,2}$ is generated by the $p$-th
root of an element $\prod_{j=0}^{d-1}E_\gamma\big(\eta^{p^j}\big)^{n_j}$, with
exponents $n_j\in\{0,1,\ldots,p-1\}$, not all zero. (In fact
the statement in \cite{Pickett} requires one element in the basis of 
$k$ over $\Z_p/p\Z_p$ to equal $1$, but this is never used in the
proof, so we can use a normal basis here instead.) To such an element
corresponds a unique $u=\sum_ju_j\eta^{p^j}\in k^\times$, where
$u_j$ is the coset of $n_j$ modulo $p\Z_p$. Let us lift each $u_j$ in $\{0\}\cup\mu_{p-1}\subset\Z_p$ and
write $u_j=n_j+pm_j$ with $m_j\in\Z_p$, then
$$x_u=\prod_{j=0}^{d-1}E_\gamma\big(\eta^{p^j}\big)^{n_j}\cdot\left(\prod_{j=0}^{d-1}E_\gamma\big(\eta^{p^j}\big)^{m_j}\right)^p\enspace.$$
We conclude that the degree $p$ extension of $K'$ contained in
$K_{p,2}$ are the $L_{(u)}=K'\big(x_u^{1/p}\big)$ for  $u\in k^\times$.

Multiplying $u$ by an element $\mu$ in $(\Z_p/p\Z_p)^\times$
changes $x_{u}$ to $x_{\mu u}={x_u}^\mu$. If we let $\mu=n+p\mu'$
with $n\in\{1,\ldots,p-1\}$ and $\mu'\in\Z_p$, we find that $x_{\mu
  u}$ equals a prime to $p$ power of $x_u$ multiplied by the $p$-th
power of an element of $K'$, so its $p$-th root generates the same
extension of $K'$ as that of $x_u$. Therefore the map given in the
statement is well defined and surjective. 
For any integer $1\le i\le r$, the compositum $M_iK'=L_i$ is a degree $p$
extension of $K'$ contained in $K_{p,2}$, so we get that the map is a
one-to-one correspondence.
\end{proof}
\begin{remark}
Keeping the notations of the proof, it would be nice to show that
$L_{(u)}$ is also generated by the $p$-th root of $E_\gamma(u)$. We
would then get a generating set (for the degree $p$ extensions of $K'$
contained in $K_{p,2}$) that would not depend on the choice of a basis
of $k$ over $\Z_p/p\Z_p$. However, the fact that $E_\gamma$ is not homomorphic
on the additive group $D(1)^+$ does not make this goal easy to achieve.
\end{remark}
\subsection{Lifting Galois automorphisms}\label{sub:lifting}
We now set $L=K'M$, thus by the former result
$L=K'(x_{\vareps}^{1/p})$ for an element
$\vareps$ of $k^\times$ which is uniquely determined modulo
$(\Z_p/p\Z_p)^\times$ by $M$. We fix $\vareps$ for the rest of the paper and we set
$x=x_{\vareps}$ for brevity. 
We describe our field extensions in 
Figure \ref{ext-dia} below,
where we let $H=\Gal(M/K)$, $\Delta=\Gal(L/M)$.  
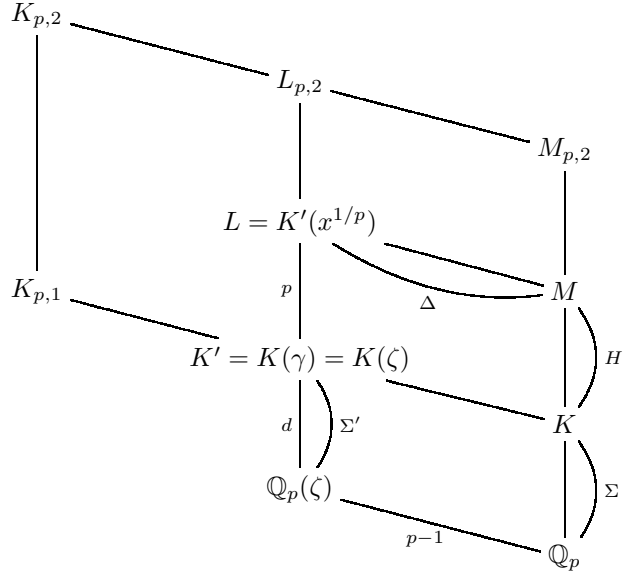
\begin{figure}[htb]
\label{ext-dia}
{
$$\xymatrix@=10pt{K_{p,2}\ar@{-}@/0pc/[drrr]\ar@{-}@/0pc/[dddd]\\
&&&L_{p,2}\ar@{-}@/0pc/[drrr]\ar@{-}@/0pc/[dd]&&&\\
&&&&&&M_{p,2}\ar@{-}@/0pc/[dd]\\
&&&L=K'(x^{1/p})\ar@{-}@/0pc/[dd]_p\ar@{-}@/0pc/[drrr]\ar@{-}@/d1pc/[drrr]_{\Delta}\\
K_{p,1}\ar@{-}@/0pc/[drrr]&&&&&&M\ar@{-}@/0pc/[dd]\ar@{-}@/r1pc/[dd]^{H}\\
&&&K'=K(\gamma)=K(\zeta)\ar@{-}@/0pc/[drrr]\ar@{-}@/r1pc/[dd]^{\Sigma'}\\
&&&&&&K\ar@{-}@/r1pc/[dd]^{\Sigma}\\
&&&\Q_p(\zeta)\ar@{-}@/0pc/[uu]^d\ar@{-}@/0pc/[drrr]_{p-1}&&&\\
&&&&&&\Q_p\ar@{-}@/0pc/[uu]}$$
}
\caption{Extensions diagram}
\end{figure}

We need to study how the elements of $\Gal(K'/K)$ and
$\Sigma'=\Gal(K'/\Q_p(\zeta))$ can be respectively lifted to
automorphisms of $L$ (recall that $L\subset K^{ab}$) and of the Galois closure $\wt L$ of
$L/\Q_p(\zeta)$. 

We have the following group isomorphisms:  
$$\begin{array}{ccccl}
\mu_{p-1} &\cong& (\Z_p/p\Z_p)^\times &\cong& \Gal(\Q_p(\zeta)/\Q_p)\cong\Gal(K'/K)\\
\mu &\longmapsto& \mu\mod p &\longmapsto& (s_\mu:\zeta\mapsto\zeta^\mu)\\
\end{array}$$
As a consequence of Formula (\ref{exp}) in Subsection
\ref{dwork}, we know that $x-1$ is a uniformising parameter of $K$. As
noted in \cite{Pickett} before Lemma 9, it implies that $x^{1/p}-1$ is
a uniformising parameter of $L$. It follows that both $x-1$ and
$x^{1/p}-1$ belong to $D(1^-)$, so we may consider raising $x$ and $x^{1/p}$ to the
$\mu$-th power for any $\mu\in\Z_p$. 
Applying Lemma 10 of \cite{Pickett} to
$x$ as well as to $\zeta$ (with appropriate choices of the exponents
$n_i$), we get that, for any $\mu\in\mu_{p-1}$:
$$s_\mu(x)=x^\mu\quad\mbox{ and }\quad s_\mu(\zeta)=\zeta^\mu\enspace.$$
Since $[L:K']=p$, each
automorphism $s_\mu\in\Gal(K'/K)$ has $p$ distinct liftings in
$\Gal(L/K)$, which are determined by their value at
$x^{1/p}$. More precisely, let us fix a $p$-th root of
$s_\mu(x)$ in $L$, by setting 
$$s_\mu(x)^{1/p}=\big(x^{1/p}\big)^\mu=x^{\mu/p}\enspace,$$
where $x^{1/p}$ is our previous (implicit) choice of a $p$-th
root of $x$. 
Any lifting $\wt s$ of $s_\mu$
satisfies $\wt s\big(x^{1/p}\big)^p=s_\mu(x)$, so there exists an integer
$n\in\{0,\ldots,p-1\}$ such that $\wt s(x^{1/p})=\zeta^{n}s_\mu(x)^{1/p}$,
and $n$ determines $\wt s$.
One easily checks that, for any integer $k$:
$${\wt s}^k(x^{1/p})=\zeta^{nk\mu^{k-1}}x^{\mu^k/p}\enspace,$$
so 
${\wt s}^{p-1}=1$ if and only if $n=0$. Further, in this 
case, one obtains that $\wt s$ has the same order in
$\Gal(L/K)$ as $\mu$ in $\mu_{p-1}$. We have thus proved:
%
\begin{proposition}\label{delta}
Let $\mu$ denote a primitive $(p-1)$-th root of unity. Then $\Gal(L/K)$
contains exactly one element $\wt s_\mu$ that maps $\zeta$ to $\zeta^\mu$
and which is of order $p-1$, hence generates $\Delta$ and fixes $M$.
Further $\wt s_\mu(x^{1/p})=x^{\mu/p}$.
\end{proposition}

We now consider extending automorphisms in $\Sigma'=\Gal(K'/\Q_p(\zeta))$. 
Since $\gamma$ is fixed by $\Sigma'$, one checks that
$\sigma(x)=x_{\sigma(\vareps)}$ for any $\sigma\in\Sigma'$. Therefore,
by Kummer theory, \(L/\Q_p(\zeta)\) is Galois if and only if every
$x_{\sigma(\vareps)}$ can be written as $x^my^p$ for some integer $m$
prime to $p$ and some $y\in K'$. Using the fact that the elements of
$\Sigma'$ act on $\eta$ by raising it to its $p^n$-th power for
$n\in\{0,1,\ldots,d-1\}$, one then checks that this only happens when
$\vareps=t\sum_{n=0}^{d-1}s^{d-n-1}\eta^{p^n}$ for some
$s,t\in(\Z_p/p\Z_p)^\times$ with $s$ of order dividing $d$. In
particular, $s=1$ yields $L=K(\xi)=K.\Q_p(\xi)$ which is always
Galois over $\Q_p(\zeta)$ --- recall $\xi$ and $\zeta$ were defined in
Equation \ref{zetaxi}. 

Since $L/\Q_p(\zeta)$ is not Galois in the general case,
we consider the Galois closure $\wt L$ of $L/\Q_p(\zeta)$,
given by $\wt L=K'(\{\sigma(x)^{1/p}:\sigma\in\Sigma'\})$. The
extension $\wt L/K'$ is Kummer and its Galois group is
$p$-elementary abelian, of order $p^m$ for some integer $m$ (equal to
$1$ if and only if $L/\Q_p(\zeta)$ is Galois). Let  
$\sigma_1=1$ and $\sigma_2,\ldots,\sigma_m\in\Sigma'$ be such that
$\wt L$ is the compositum of the $m$ degree $p$ extensions
$K'(\sigma_n(x)^{1/p})$, $1\le n\le m$, of $K'$. Any
$\sigma\in\Sigma'$ extends to $\Gal(\wt L/\Q_p(\zeta))$ in $p^m=[\wt
  L:K']$ different ways, determined by the values at the
$\sigma_n(x)^{1/p}$, $1\le n\le m$. More precisely, let us fix a $p$-th root of
$\sigma\sigma_n(x)$ for each $n\in\{1,\ldots,m\}$, then a
lifting $\wt\sigma$ of $\sigma$ is determined by the integers
$k(n)\in\{0,1,\ldots,p-1\}$ such that, for any $n$,
$\wt\sigma\big(\sigma_n(x)^{1/p}\big)=\zeta^{k(n)}\big(\sigma\sigma_n(x)\big)^{1/p}$. The 
choices of the $k(n)$ for all the $n$ yield the $p^m$
possible liftings of $\sigma$, hence each of these choices is realised.
In particular, taking $k(1)=0$, we
see that there exists a lifting $\wt\sigma$ of $\sigma$ to $\wt L$
such that
\begin{equation}\label{lifting}
\wt\sigma(x^{1/p})=\sigma(x)^{1/p}\enspace,
\end{equation}
for any prior choice of a $p$-th root of $\sigma(x)$. We deduce the
following result. Let $N_{K'/\Q_p(\zeta)}(x)$ denote the norm of
$x$ from $K'$ to $\Q_p(\zeta)$.
\begin{proposition}\label{transversal}
For any choice of a $p$-th root of $N_{K'/\Q_p(\zeta)}(x)$, there
exists a transversal $\Omega$ of $\Omega_K$ in $\Omega_{\Q_p}$ 
such that
$$\prod_{\omega\in\Omega}\big(x^{1/p}\big)^\omega=N_{K'/\Q_p(\zeta)}(x)^{1/p}\enspace.$$
\end{proposition}
\begin{proof}
Since $K/\Q_p$ is Galois, choosing a transversal $\Omega$ of
$\Omega_K$ in $\Omega_{\Q_p}$ is the same as choosing a way to extend
the elements of $\Sigma=\Gal(K/\Q_p)$ to act on $\Q_p^c$. By Galois theory,
we know that there is only one way to extend them to $\Sigma'$. Let
$\sigma\in\Sigma'$, then for any choice of a $p$-th root of
$\sigma(x)$, there exists $\omega_\sigma\in\Omega_{\Q_p}$ that extends
the lifting $\wt\sigma$ of $\sigma$ defined in (\ref{lifting}), namely:
$$\big(x^{1/p}\big)^{\omega_\sigma}=\wt\sigma(x^{1/p})=\sigma(x)^{1/p}\enspace.$$
The set $\Omega=\{\omega_\sigma,\sigma\in\Sigma'\}$ defined this way
is a transversal of $\Omega_K$ in $\Omega_{\Q_p}$, and 
$$\prod_{\omega\in\Omega}\big(x^{1/p}\big)^\omega=\prod_{\sigma\in\Sigma'}\sigma(x)^{1/p}\enspace,$$
which can be made to equal any $p$-th root of
$N_{K'/\Q_p(\zeta)}(x)$ for an adequate choice of the $p$-th roots
of the $\sigma(x)$, $\sigma\in\Sigma'$.
\end{proof}
%
\subsection{The Norm-Resolvent}\label{sub:normres}
We begin by exhibiting a (self-dual) normal basis
generator for the square root of the inverse different $\A_{M/K}$,
which we then use to compute the norm-resolvent involved in $\RT_K(M)$.
Recall $x$ has been defined at the beginning of Subsection \ref{sub:lifting}.
\begin{lemma}\label{basis_gen_lemma} 
Let
\[\alpha_{M}=\frac{1+Tr_{\Delta}(x^{1/p})}{p}\enspace,\]
then $\alpha_M$ is a self-dual normal basis generator for
 $\A_{M/K}$.
\end{lemma}
\begin{proof}This Lemma is a consequence of \cite[Theorem 12]{Pickett}.
\end{proof}

The extension $L/K'$ is
Kummer with generator $x^{1/p}$ so its Galois group is generated by
the automorphism defined by $x^{1/p}\mapsto\zeta x^{1/p}$. Further
by Galois theory the restriction of this automorphism to $M$
generates $H=\Gal(M/K)$. This enables us to fix a generator $h$ of $H$ and
the irreducible character $\chi$ of $\Gal(M/K)$ such that $M$ is
the fixed subfield of $M_{p,2}$ by $\ker(\chi)$ (this
condition only determines $\chi$ up to a prime to $p$ power).
\begin{notation}\label{gchi}
let $\wt h\in\Gal(L/K')$ be such
that $\wt h\big(x^{1/p}\big)=\zeta x^{1/p}$ and set $h=\wt h|_{M}$.
Let $\chi$ be the irreducible character of $\Gal(M_{p,2}/K)$ such that
$\ker(\chi)=\Gal(M_{p,2}/M)$ and $\chi(h)=\zeta$. 
\end{notation}
We can now state the theorem that we will prove in this
subsection. Recall the notations from Equation \ref{zetaxi} and that
$Tr=Tr_{K/\Q_p}$.
\begin{theorem}\label{norm_resolve_theorem}
There exists a choice of the transversal
$\Omega$ defining the norm-resolvent such that
$\mathcal{N}_{K/\Q_p}(\alpha_{M}\mid\chi_0)=1$ and, for any $j\in\{1,\ldots,p-1\}$:
\[\mathcal{N}_{K/\Q_p}(\alpha_{M}\mid\chi^j)=\xi^{j(2-j^{1-p})Tr(\vareps)}\enspace,\]
where $j^{1-p}$ is the inverse of $j^{p-1}$ in $\Z_p$. 
\end{theorem}

Before we can prove this theorem we must calculate the properties of
certain elements. 
We begin with the computation of the norm $N_{K'/\Q_p(\zeta)}(x)$, which
establishes a new link between $x$ and $\vareps$. 
\begin{lemma}\label{xiui}
$N_{K'/\Q_p(\zeta)}(x)=\zeta^{Tr(\vareps)}$.
\end{lemma}
\begin{proof}
Recall that $\Sigma'=\Gal\big(K'/\Q_p(\zeta)\big)$ fixes
$\gamma\in\Q_p(\zeta)$ so, if $\sigma\in\Sigma'$ and $u\in\bo_K$,
$\sigma\big(E_\gamma(u)\big)=E_\gamma\big(\sigma(u)\big)$. Further
$\Sigma'$ acts on $\eta$ by raising it to the power $p^n$ for some
$0\le n\le d-1$, and fixes $\mu_{p-1}$. Therefore, writing
$\vareps=\sum_je_j\eta^{p^j}$ with coefficients $e_j\in\Z_p/p\Z_p$:
$$N_{K'/\Q_p(\zeta)}(x)=\prod_{n=0}^{d-1}\,\prod_{j=0}^{d-1}E_\gamma\big(\eta^{p^{n+j}}\big)^{e_j}\enspace.$$
For each $0\le j,n\le d-1$, consider the power series
$E_\gamma\big(\eta^{p^{n+j}}X\big)$, obtained by substituting $\eta^{p^{n+j}}X$
for $X$ in $E_\gamma$; it converges on $D(1^+)$ and
\begin{align*}
\prod_{n=0}^{d-1}E_\gamma\big(\eta^{p^{n+j}}X\big)
&= \exp\Big(\sum_{n=0}^{d-1}\big(\gamma \eta^{p^{n+j}}X-\gamma \eta^{p^{n+j+1}}X^p\big)\Big)\\
&= \exp\big(Tr(\eta)(\gamma X-\gamma X^p)\big)\\
&= E_\gamma(X)^{Tr(\eta)}\enspace,
\end{align*}
that also converges on $D(1^+)$. Evaluating at $X=1$ and using the
above formula for the norm yields the result, since $Tr(\vareps)=Tr(\eta)\sum_je_j$.
\end{proof}

In view of Proposition \ref{transversal}, we now fix our transversal
$\Omega$ of $\Omega_K$ in $\Omega_{\Q_p}$ so that the product of the
$\Omega$-conjugates of $x^{1/p}$ equals $\xi^{Tr(\vareps)}$, which is a
$p$-th root of $N_{K'/\Q_p(\zeta)}(x)$ by the preceding Lemma. We get:
\begin{lemma}\label{norm_lemma}
$\prod_{\omega\in\Omega}\left(x^{1/p}\right)^\omega=\xi^{Tr(\vareps)}$.
\end{lemma}
\smallskip

Recall that $\Gal(M/K)=H=\langle h\rangle$, where $h$ is as in
Notation \ref{gchi}. Let $\Gal(L/M)=\Delta=\langle\delta\rangle$, then 
by Proposition \ref{delta}, $\delta=\wt s_\mu$ for some primitive
$(p-1)$-th root of unity $\mu$, and $\delta(\zeta)=\zeta^\mu$,
$\delta(x^{1/p})=x^{\mu/p}$. 
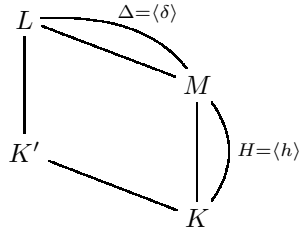
$$\xymatrix@=12pt{
{L}\ar@{-}@/0pc/[dd]\ar@{-}@/0pc/[drrr]\ar@{-}@/ur1pc/[drrr]^{\Delta=\langle\delta\rangle}\\
& & & M\ar@{-}@/0pc/[dd]\ar@{-}@/r1pc/[dd]^{H=\langle h\rangle}\\
K'\ar@{-}@/0pc/[drrr]\\
& & & K}$$
We now compute the resolvent for the normal basis generator
$\alpha_{M}$ that was defined in Lemma \ref{basis_gen_lemma}. 
\begin{proposition}
One has $(\alpha_M\mid\chi_0)=1$ and, for $j\in\{1,\ldots,p-1\}$:
\[(\alpha_{M}\mid\chi^j)=(x^{1/p})^{\mu^s}\enspace,\]
where $0\leq s\leq p-2$ is such that $\mu^s\equiv j\mod p$.
\end{proposition}
\begin{proof}
The definitions of the resolvent and of $\alpha_M$ yield:
$$(\alpha_{M}\mid\chi^j)=\sum\limits_{t=0}^{p-1}\frac{1+h^t\big(Tr_{\Delta}(x^{1/p})\big)}{p}\,\chi^j(h^{-t})\enspace.$$
Let $t\in\{0,1,\ldots,p-1\}$, then:
$$h^t\big(Tr_{\Delta}(x^{1/p})\big)=\sum_{s=0}^{p-2}\wt
h^t\left(\big(x^{1/p}\big)^{\mu^s}\right)=\sum_{s=0}^{p-2}\zeta^{t\mu^s}\big(x^{1/p}\big)^{\mu^s}\enspace.$$

If $j=0$, then $\chi^j$ is the trivial character $\chi_0$, so that:
$$(\alpha_{M}\mid\chi_0)=1+\frac{1}{p}\sum_{t=0}^{p-1}\sum_{s=0}^{p-2}\zeta^{{t\mu^s}}\big(x^{1/p}\big)^{\mu^s}
=1+\frac{1}{p}\sum_{s=0}^{p-2}\left(\sum_{t=0}^{p-1}\zeta^{{t\mu^s}}\right)\big(x^{1/p}\big)^{\mu^s}$$
and $\sum\limits_{t=0}^{p-1}\zeta^{{t\mu^s}}=0$, so $(\alpha_{M}\mid\chi_0)=1$.

We now assume $j\neq 0$. Since $\chi^j(h)=\zeta^j$, hence
$\sum\limits_{t=0}^{p-1}\chi^j(h^{-t})=0$, we get
$$(\alpha_{M}\mid\chi^j)=\frac{1}{p}\sum_{s=0}^{p-2}\left(\sum_{t=0}^{p-1}\zeta^{{t(\mu^s}-j)}\right)\big(x^{1/p}\big)^{\mu^s}\enspace;$$
we observe that
\[\sum_{t=0}^{p-1}\zeta^{{t(\mu^s}-j)}=\left\{
\begin{array}{ll}
p&\text{if $\mu^s\equiv j\mod p$,}\\
0&\text{otherwise,}
\end{array}\right.\]
hence the result.
\end{proof}

We are now in a position to prove Theorem \ref{norm_resolve_theorem}.

 \begin{proof}[Proof of Theorem \ref{norm_resolve_theorem}]
When $j=0$ the result is clear. We now assume that $j\neq 0$. Recall
the choice we made before Lemma \ref{norm_lemma} for our transverse
$\Omega$ of $\Omega_K$ in $\Omega_{\Q_p}$. Since
$\chi^j$ takes values in $\Q_p(\zeta)$, which is fixed by $\Omega$,
Definition \ref{resnormres} of the norm-resolvent yields:
\begin{align*}
\mathcal{N}_{K/\Q_p}(\alpha_{M}\mid\chi^j)&=
\prod\nolimits_{\Omega}(\alpha_{M}\mid\chi^j)^{\omega}
=\prod\nolimits_{\Omega}\big((x^{1/p})^{\mu^s}\big)^\omega\\
&=\left(\prod\nolimits_{\Omega}(x^{1/p})^{\omega}\right)^{\mu^s}
=\xi^{\mu^sTr(\vareps)}\\
\end{align*} 
using Lemma \ref{norm_lemma}. Writing $\mu^s\equiv j+ap\mod p^2$ for some $a\in\{0,1,\ldots,p-1\}$
and raising to the $(p-1)$-th power yields $1\equiv
j^{p-1}-apj^{p-2}\mod p^2$, thus:
$$ap\equiv(j^{p-1}-1)j^{2-p}\equiv j-j^{2-p}\mod p^2\enspace.$$
It follows that $\mu^s\equiv
j(2-j^{1-p})\mod p^2$, which ends the proof of Theorem \ref{norm_resolve_theorem}.
\end{proof}
%
%
\subsection{The modified twisted Galois Gauss Sum}
Recall from Notation \ref{gchi} that $\chi$ is the character of
$\Gal(M_{p,2}/K)$ such that $M$ is the fixed
field of $\ker(\chi)$ and $\chi(h)=\zeta$ for our choice of generator
$h$ of $H=\Gal(M/K)$. Our character $\chi$ is weakly ramified so we
know by Corollary \ref{tauQ} that there exists $v_\chi\in\bo_K^\times$
such that: 
\begin{equation}\label{condition}
\forall\,u\in\bo_K\ ,\quad\chi(1+up)^{-1}=\zeta^{Tr(uv_\chi)}\enspace. 
\end{equation} 
We are going to show that $v_\chi$ can be chosen so that its trace
from $K$ to $\Q_p$ equals that of $\vareps$. In order to do that, we
need some properties of the $p$-th Hilbert symbol (see 
\cite[Ch.IV]{Fesenko-Vostokov}). We have $\car(K')=0$ and $\zeta\in
K'$. Let $\mu_p=\langle\zeta\rangle$ denote the group of $p$-th roots
of unity in $\Q_p^c$.
The $p$-th Hilbert symbol of $K'$ is defined as 
$$\begin{array}{rccl}
(\,\ ,\ )_{p,K'}:&K'^{{\times}}{\times} K'^{{\times}}&\longrightarrow& \mu_p \\ 
&(a,b)&\longmapsto& \dfrac{\theta_{K'}(a)(b^{1/p})}{b^{1/p}}\enspace.
\end{array}$$
\begin{proposition}\label{Hilbert_chi_lemma}
For all $u\in\bo_K$, we have:
\[(1+up,x)_{p,K'}=\chi(1+up)^{-1}\enspace.\]
\end{proposition}
\begin{proof}
The proof is in several steps. Let $L_{p,2}$ be the compositum of the
fields $L_i$ for $i\in\{1,\ldots,r\}$.
Recall that we identify the residue field
$k=\{0\}\cup k^\times$ with $\{0\}\cup\mu_{q-1}\subset\bo_K$ through Teichm\"uller's
lifting. We first show:   
\begin{lemma}\label{gal}
$\Gal(L_{p,2}/K')=\theta_{L_{p,2}/K'}(U_K^1/U_K^2)=\{\theta_{L_{p,2}/K'}(1+up):u\in k\}$.
\end{lemma}
\begin{proof}
First note that $L_{p,2}\subset K_{p,2}$, so $\theta_{L_{p,2}/K}$ is
trivial on $U_K^2$ and the same holds for $\theta_{L_{p,2}/K'}$ (for
instance using \cite[Theorem 6.16]{iwasawa}). Hence we may consider
$\theta_{L_{p,2}/K'}(U_K^1/U_K^2)$, which as a set clearly equals
$\{\theta_{L_{p,2}/K'}(1+up):u\in k\}$. 

By local class field theory, $\Gal(M_{p,2}/K)=\theta_{M_{p,2}/K}(U_K^1)$ and
the intersection of the kernel of $\theta_{M_{p,2}/K}$ with $U_K^1$ is $U_K^2$
(since $M_{p,2}\subset K_{p,2}$ and $[M:K]=q$). It follows that
$\Gal(M_{p,2}/K)=\theta_{M_{p,2}/K}(U_K^1/U_K^2)$.

Since $L_{p,2}=M_{p,2}K'$ with $K'/K$ and $M_{p,2}/K$ linearly
disjoint, the functorial properties of the Artin reciprocity map 
yield \cite[Theorem 6.9]{iwasawa}:
\begin{equation}\label{theta}
\theta_{L_{p,2}/K'}\big|_{M_{p,2}}=\theta_{M_{p,2}/K}\circ
N_{K'/K}\enspace,
\end{equation}
where $N_{K'/K}$ stands for the norm from $K'$ to $K$.
For $u\in k$ we have $1+up\in\bo_K$, so
$N_{K'/K}(1+up)=(1+up)^{p-1}\equiv 1-up\mod p^2\bo_K$. We get that $N_{K'/K}$ is
an isomorphism from $U_K^1/U_K^2$ into itself, and therefore
$$\theta_{L_{p,2}/K'}(U_K^1/U_K^2)\big|_{M_{p,2}}=\Gal(M_{p,2}/K)\enspace.$$
This yields the result using Galois theory, since the restriction map
$g\mapsto g|_{M_{p,2}}$ is an isomorphism from $\Gal(L_{p,2}/K')$ to $\Gal(M_{p,2}/K)$.
\end{proof}
\begin{lemma}
There exists $t\in\{1,\ldots,p-1\}$ such that, for all $u\in k$, 
$$(1+up,x)^t_{p,K'}=\chi(1+up)^{-1}\enspace.$$
\end{lemma}
\begin{proof}
By definition, $\chi(1-up)=1$ if and only if $\theta_{M_{p,2}/K}(1-up)$
fixes $M$. This is in turn equivalent to $\theta_{L_{p,2}/K'}(1+up)$
fixing $L$, since $L=MK'$ and we know by (\ref{theta}) that
$\theta_{L_{p,2}/K'}(1+up)$ is the only lifting of 
$\theta_{M_{p,2}/K}(1-up)$ to $L_{p,2}/K'$. Since $L=K'(x^{1/p})$ and by
definition of the Hilbert symbol, we get:
$$\chi(1-up)=1\Leftrightarrow(1+up,x)_{p,K'}=1\enspace.$$
The properties of the Hilbert symbol and the fact that
$\theta_{L_{p,2}/K'}$ is trivial on $U_K^2$ give us
\begin{align*}
(1+up,x)_{p,K'}(1+u'p,x)_{p,K'}&=(1+up+u'p+uu'p^2,x)_{p,K'}\\
&=(1+(u+u')p,x)_{p,K'}
\end{align*}
for $u,u'\in k$, which means that $u\mapsto(1+up,x)_{p,K'}$ is a
character of the additive group of $k$. We also know that
$u\mapsto\chi(1-up)$ is a character of the additive
group of $k$ (since $u\mapsto\theta_{M_{p,2}/K}(1-up)$ is). Therefore $u\mapsto(1+up,x)_{p,K'}$ and
$u\mapsto\chi(1-up)$ are characters of the same $p$-elementary abelian
group which have the same kernel of index $p$. It follows that
$(1+up,x)^t_{p,K'}=\chi(1-up)=\chi(1+up)^{-1}$ for some $t$. 
\end{proof}

To finish the proof of Proposition \ref{Hilbert_chi_lemma}, note that
the arguments in the proof of Lemma \ref{gal} can be adjusted to show
that $\Gal(L/K')=\langle\theta_{L/K'}(1+ap)\rangle$ for an adequate
element $a\in\bo_K$. Hence
there exists an integer $n\in\{1,\ldots, p-1\}$ such that 
$$\wt h=\theta_{L/K'}(1+ap)^n=\theta_{L/K'}\big((1+ap)^n\big)=\theta_{L/K'}\big(1+wp+p^2b\big)\enspace,$$
for some $b\in\bo_K$, where we let $w\in\mu_{q-1}$ be such that
$w\equiv na\mod p\Z_p$.
Since
$\theta_{L/K'}\big(1+wp+p^2b\big)|_{M}=\theta_{M/K}(1-wp)=\theta_{L/K'}(1+wp)|_{M}$,
we get that
$$\wt h=\theta_{L/K'}(1+wp)\quad\mbox{ and }\quad h=\wt h|_{M}=\theta_{M/K}(1-wp)\enspace,$$
hence 
$$\chi(1+wp)^{-1}=\chi(1-wp)=\chi(h)=\zeta=\frac{\wt
  h(x^{1/p})}{x^{1/p}}=(1+wp,x)_{p,K'}\enspace,$$
which implies $t=1$ in the preceding lemma.
\end{proof} 

We can now show the announced result, recalling that $Tr=Tr_{K/\Q_p}$.
\begin{corollary}\label{u=v}
There exists $v_\chi\in\bo_K^\times$ such that $v_\chi$ satisfies 
condition (\ref{condition}) from Corollary \ref{tauQ} for $\chi$ and
$Tr(v_\chi)=Tr(\vareps)$. 
\end{corollary}
\begin{proof}
Throughout this proof we fix $u\in\mathbb{Z}_p/p\mathbb{Z}_p$.
We let $\ver:\Omega_{\Q_p}^{ab}\rightarrow\Omega_{\Q_p(\zeta)}^{ab}$
be the transfer map from $\Q_p$ to $\Q_p(\zeta)$. From \cite[Theorem
  6.16 and Formula (3) p.93]{iwasawa}, where $\ver$ is denoted as $t_{\Q_p(\zeta)/\Q_p}$,
we know that:
$$\theta_{\Q_p(\zeta)}(1+up)=\ver\big(\theta_{\Q_p}(1+up)\big)=\prod_{\tau\in\Gal(\Q_p(\zeta)/\Q_p)}\tau\,\theta_{\Q_p}(1+up)\,\tau^{-1}\enspace.$$
As $\Gal(\Q_p(\zeta)/\Q_p)$ and $\Omega_{\Q_p}^{ab}$ commute, we get
$\theta_{\Q_p(\zeta)}(1+up)=\theta_{\Q_p}(1+up)^{p-1}$.
By \cite[\S3.1 Remark after Theorem 2]{Serre-lubintate}, we know
that $\theta_{\Q_p}(1+up)(\xi)=\xi^{1-up}$, so that:
\begin{equation}\label{theta_xi_equation}
\theta_{\Q_p(\zeta)}(1+up)(\xi)=\xi^{1+up}\enspace.
\end{equation}
Using the properties of the Hilbert symbol we make the following derivation.
\begin{align*}
(1+up,x)_{p,K'}&=(1+up,N_{K'/\Q_p(\zeta)}(x))_{p,\Q_p(\zeta)}&\text{(from \cite[IV\S5]{Fesenko-Vostokov})}\\
&=\big(1+up,\zeta^{Tr(\vareps)}\big)_{p,\Q_p(\zeta)}&\text{(from
  Lemma \ref{xiui})}\\
&=\big(1+up,\zeta\big)_{p,\Q_p(\zeta)}^{Tr(\vareps)}
=\left(\frac{\theta_{\Q_p(\zeta)}(1+up)(\xi)}{\xi}\right)^{Tr(\vareps)}&\\
&=\left(\frac{\xi^{1+up}}{\xi}\right)^{Tr(\vareps)}&\text{(from Equation \ref{theta_xi_equation})}\\
&=(\xi^{up})^{Tr(\vareps)}=\zeta^{uTr(\vareps)}&
\end{align*}

On the other hand, let $v\in\bo_K^\times$ satisfying condition
(\ref{condition}). Then from Proposition \ref{Hilbert_chi_lemma} 
we know that $(1+up,x)_{p,K'}=\chi(1+up)^{-1}=\zeta^{Tr(uv)}$, and so we have 
$$(1+up,x)_{p,K'}=\zeta^{uTr(v)}\enspace.$$
Comparing with the former equality, this yields $Tr(\vareps)\equiv
Tr(v)\mod p\Z_p$, so let $a\in\Z_p$ such that
$Tr(\vareps)=Tr(v)+pa$. Since $K/\Q_p$ is unramified, there exists
$b\in\bo_K$ such that $Tr(b)=a$, so $Tr(\vareps)=Tr(v+pb)=Tr(v_\chi)$ if we
let $v_\chi=v+pb$, which proves the result using $(i)$ in Corollary \ref{tauQ}.
\end{proof}

We deduce the following expression for the modified twisted Galois
Gauss sum, using the statement and property $(ii)$ of Corollary
\ref{tauQ}. Note that since $p\not=2$,
$\frac{p^2}{4v_\chi}\bo_K=\pi^2\D_K$, so we may set $c_{K,2}=p^2/4v_\chi$.
\begin{theorem}\label{gauss_sum_theorem}
Let $v_\chi$ be as in Corollary \ref{u=v} and set $c_{K,2}=p^2/4v_\chi$. Then
$\tau_K^\star(\chi_0-\chi_0^{2})=1$ and, for any
$j\in\{1,\ldots,p-1\}$:   
\[\tau_K^\star(\chi^j-\chi^{2j})=\chi^j(j^{-1})\xi^{-jTr({\vareps})}\enspace. \] 
\end{theorem}

The dependency relationships between our constants might look
complicated, so let us try to sum up how we fixed them. Our primitive
$p$-th root of unity $\zeta$ came first; the extension $M/K$ under
study determined a unit $\vareps$ up to $(\Z_p/p\Z_p)^\times$; we defined a Kummer generator $x$, then a
generator $h$ of $H=\Gal(M/K)$, and then a generator 
$\chi$ of $\wh H$; with $\chi$ came the unit $v_\chi$, but only modulo
$p\bo_K$; the knowledge of $\vareps$ enabled us to fix $v_\chi$ in Corollary
\ref{u=v} and finally $c_{K,2}$ in Theorem \ref{gauss_sum_theorem}. 

Apart from the
dependency upon the choice of $\zeta$, which is shared by the usual
Galois Gauss sum, our modified Galois Gauss sum thus also depends on $M$. This
does not prevent deducing Theorem \ref{THM} from Theorem \ref{local}
since only one extension $M_i/F_\wp$ has to be considered at each
wildly and weakly ramified prime ideal $\wp$ of $\bo$. 
\subsection{The product}
We can now end the proof of Theorem \ref{local}.
By Theorem \ref{norm_resolve_theorem} and Theorem
\ref{gauss_sum_theorem}, the product of our norm-resolvent and modified twisted Galois Gauss sum is $1$ when evaluated at the trivial character and we have,
for $j\in\{1,\ldots,p-1\}$: 
$$\mathcal{N}_{K/\Q_p}(\alpha_{M}\mid\chi^j)\tau_K^\star(\chi^j-\chi^{2j})
=\left(\chi(j^{-1})\zeta^{Tr(\vareps)(1-j^{1-p})/p}\right)^j\enspace.$$
Note that $1-j^{1-p}\in p\Z_p$; we now wish to show that this
equals $1$. We are thus left with showing
that $\chi(j)=\zeta^{Tr(\vareps)(1-j^{1-p})/p}$, which is equivalent to
$\theta_{M/K}(j)=h^{Tr(\vareps)(1-j^{1-p})/p}$, since $h\in H$ is such
that $\chi(h)=\zeta$. It is also equivalent to showing
\[\theta_{M/K}(j)^{p-1}=h^{Tr(\vareps)(j^{1-p}-1)/p}\enspace.\]
In order to shift this relation to
$\Gal(L/K')$, we notice that
$\theta_{M/K}(j)^{p-1}=\theta_{M/K}(N_{K'/K}(j))=\theta_{L/K'}(j)|_{M}$
and recall that $\wt h|_{M}=h$. Thus the relation holds if and only if:
$$\theta_{L/K'}(j)=\wt h^{\,Tr(\vareps)(j^{1-p}-1)/p}\enspace.$$
We can now evaluate these automorphisms at
$x^{1/p}$, recalling that $\wt h(x^{1/p})=\zeta x^{1/p}$ and
$(j,x)_{p,K'}=\frac{\theta_{L/K'}(j)(x^{1/p})}{x^{1/p}}$, so
we are left with proving:
$$(j,x)_{p,K'}=\zeta^{\,Tr(\vareps)(j^{1-p}-1)/p}\enspace.$$
Using the properties of the Hilbert symbol we get:
$$(j,x)_{p,K'}=\big(j,N_{K'/\Q_p(\zeta)}(x)\big)_{p,\Q_p(\zeta)}=\big(j,\zeta\big)_{p,\Q_p(\zeta)}^{Tr(\vareps)}=\left(\frac{\theta_{\Q_p(\zeta)}(j)(\xi)}{\xi}\right)^{Tr(\vareps)}\enspace.$$
From \cite[\S3.1 Remark after Theorem 2]{Serre-lubintate}
we know that $\theta_{\Q_p}(j)(\xi)=\xi^{-j}$. Therefore, reasoning
as in the proof of Corollary \ref{u=v}, we get 
$$\theta_{\Q_p(\zeta)}(j)(\xi)=\big(\ver\theta_{\Q_p}(j)\big)(\xi)=\theta_{\Q_p}(j)^{p-1}(\xi)=\xi^{j^{-(p-1)}}=\xi^{j^{1-p}}\enspace,$$
hence 
$$(j,x)_{p,K'}=\left(\xi^{j^{1-p}-1}\right)^{Tr(\vareps)}=\left(\zeta^{(j^{1-p}-1)/p}\right)^{Tr(\vareps)}$$
as desired.

This ends the proof of Theorem \ref{local}, hence also of Theorem \ref{THM}.

%
 
\bibliography{bib}
\end{document}